\documentclass[11pt, reqno]{lieop}

\usepackage[latin1]{inputenc}
\usepackage{amsmath, amsthm, amscd, amsfonts, amssymb, latexsym, graphicx, color, stmaryrd}
\usepackage{mathrsfs}
\usepackage{eurosym,hyperref}
\usepackage{fancyhdr}
\usepackage{tikz-cd}
\usepackage[all,cmtip]{xy}	

\tikzset{%
    symbol/.style={%
        ,draw=none
        ,every to/.append style={%
            edge node={node [sloped, allow upside down, auto=false]{$#1$}}}
    }
}

\title{\textbf{Boutet de Monvel operators on Lie manifolds with boundary \footnote{\textbf{email:} bohlen.karsten@math.uni-hannover.de}}}
\author{Karsten Bohlen}

\def\presuper#1#2%
  {\mathop{}%
   \mathopen{\vphantom{#2}}^{#1}%
   \kern-\scriptspace%
   #2}

\newcommand{\iso}{\xrightarrow{\sim}}	

\newcommand{\Nn}{\mathbb{N}}
\newcommand{\Rr}{\mathbb{R}}
\newcommand{\Cc}{\mathbb{C}}
\newcommand{\Zz}{\mathbb{Z}}
\newcommand{\B}{\mathcal{B}}
\newcommand{\Bc}{\overline{\B}}	

\renewcommand{\H}{\mathcal{H}} 
\newcommand{\C}{\mathcal{C}}	
\newcommand{\D}{\mathcal{D}}	

\newcommand{\F}{\mathcal{F}} 
\newcommand{\Ff}{\F_{\mathrm{f}}}	
\newcommand{\G}{\mathcal{G}}	
\newcommand{\CG}{\C \G} 	
\newcommand{\Gop}{\mathrm{\G^{(0)}}} 
\newcommand{\Gmor}{\mathrm{\G^{(1)}}} 
\newcommand{\Gpull}{\mathrm{\G^{(2)}}}	

\newcommand{\Hop}{\mathrm{\H^{(0)}}}

\newcommand{\J}{\mathcal{J}}		
\newcommand{\K}{\mathcal{K}}	

\renewcommand{\L}{\mathcal{L}} 
\newcommand{\R}{\mathcal{R}}
\newcommand{\SV}{\mathcal{S}^{\V}}	
\newcommand{\X}{\mathcal{X}}	
\newcommand{\Xop}{\X^t}
	% ...

\newcommand{\V}{\mathcal{V}}		
\newcommand{\W}{\mathcal{W}}		
\newcommand{\A}{\mathcal{A}}		
\newcommand{\N}{\mathcal{N}}		
\newcommand{\End}{\mathrm{End}}		
\newcommand{\Hom}{\mathrm{Hom}}		
\newcommand{\U}{\mathcal{U}}		
\renewcommand{\P}{\mathcal{P}}		
\newcommand{\Diff}{\mathrm{Diff}}	
\newcommand{\supp}{\mathrm{supp}}	
\newcommand{\Green}{\mathcal{G}}  	
\newcommand{\Trace}{\mathscr{T}}	
\newcommand{\Poisson}{\mathcal{K}}	
\newcommand{\Pot}{\Poisson}

\newcommand{\scal}[2]{\langle #1, #2 \rangle}	
\newcommand{\ideal}[1]{\langle #1 \rangle} 	
\newcommand{\im}{\operatorname{im}} 
\newcommand{\coker}{\operatorname{coker}}	
\newcommand{\op}{\operatorname{op}} 
\newcommand{\piotimes}{\hat{\otimes}} 
\newcommand{\ind}{\operatorname{ind}}	
\newcommand{\id}{\operatorname{id}}	

\newcommand{\OpV}{\operatorname{Op}_{\V}}
\newcommand{\flip}{\operatorname{f}}
\newcommand{\singsupp}{\operatorname{singsupp}}

\newcommand{\diag}{\operatorname{diag}}

\newtheorem{Thm}{Theorem}[section]

\newtheorem{Lem}[Thm]{Lemma}
\newtheorem{Prop}[Thm]{Proposition}

\theoremstyle{definition}
\newtheorem{Def}[Thm]{Definition}
\newtheorem{Not}[Thm]{Notation}
\newtheorem{Rem}[Thm]{Remark}
\newtheorem*{Proof}{Proof}

\theoremstyle{remark}
\newtheorem{Exa}[Thm]{Example}

\begin{document}

\maketitle

\section*{Abstract}
We introduce and study a general pseudodifferential calculus for boundary value problems on a class of non-compact
manifolds with boundary (so-called Lie manifolds with boundary). 
This is accomplished by constructing a suitable generalization of the Boutet de Monvel calculus for boundary value problems.
The data consists of a compact manifold with corners $M$ that is endowed with a Lie structure of vector fields $2 \V$, a so-called Lie manifold.
The manifold $M$ is split into two equal parts $X_{+}$ and $X_{-}$ which intersect in an embedded hypersurface $Y \subset X_{\pm}$. 
Our goal is to describe a transmission Boutet de Monvel calculus for boundary value problems compatible with the structure of Lie manifolds.
Starting with the example of $b$-vector fields, we show that there are two groupoids integrating the Lie structures on $M$ and on $Y$,
respectively. These two groupoids form a bibundle (or a groupoid correspondence) and, under some mild assumptions, these groupoids
are Morita equivalent. 
With the help of the bibundle structure and canonically defined manifolds with corners, which are blow-ups in particular cases, we define a class of 
Boutet de Monvel type operators. 
We then define the representation homomorphism for these operators and show closedness under composition with the help of a representation theorem.
Finally, we consider appropriate Fredholm conditions and construct the parametrices for elliptic operators in the calculus.

\noindent \textbf{Keywords:} Boutet de Monvel's calculus, groupoid, Lie manifold.


\newpage



\newpage


\section{Introduction}

\label{intro}

In this work we will enlarge the groupoid pseudodifferential calculus introduced in \cite{NWX} and develop a general notion of a pseudodifferential calculus for boundary 
value problems in the framework of Lie groupoids. The most natural approach seems to be along the lines of the Boutet de Monvel calculus. 
Boutet de Monvel's calculus (e.g. \cite{BM}) was established in 1971. 
This calculus provides a convenient and general framework to study the classical boundary value problems.
At the same time parametrices are contained in the calculus and it is closed under composition of elements.

\subsection{Overview}

In our case, consider the following data: a Lie manifold $(X, \V)$ with boundary $Y$, which is an embedded, transversal hypersurface $Y \subset X$ in the compact manifold
with corners $X$ and which is a Lie submanifold of $X$ (cf. \cite{ALN}, \cite{AIN}). The \emph{Lie structure} $\V \subset \Gamma(TX)$ is a Lie algebra of smooth vector fields 
such that $\V$ is a subset of the Lie algebra $\V_b$ of all vector fields tangent to the boundary strata and a finitely generated  projective $C^{\infty}(X)$-module. From $X$ we define the double $M = 2X$ at the hypersurface $Y$ which is a Lie manifold $(M, 2\V)$.
The corresponding Lie structure $2\V$ on $M$ is such that $\V = \{V_{|X} : V \in 2\V\}$.   
Transversality of $Y$ in relation to $M$ means that for each given hyperface $F \subset M$ we have
\begin{align}
T_x M &= T_x F + T_x Y, \ x \in Y \cap F. \label{transv}
\end{align}
Introduce the following notation for interior and boundary: by $\partial M$ we mean the union of all hyperfaces of the
manifold with corners $M$,
\[
M_0 := M \setminus \partial M, \ Y_0 := Y \cap M_0, \ X_0 := X \cap M_0 \ \text{and} \ \partial Y := Y \cap \partial M. 
\]
For an open hyperface $F$ in $M$ we denote by $\overline{F}$ the closure in $M$. 
Denote by $\partial_{reg} F = \partial_{reg} \overline{F} = \overline{F} \cap Y$ the \emph{regular boundary} of $F$.
The hypersurface $Y$ is endowed with a Lie structure as in \cite{AIN}:
\begin{align}
\W = \{V_{|Y} : V \in 2\V, V_{|Y} \ \text{tangent to} \ Y\}. \label{bdylie}
\end{align}

We make the following \emph{assumptions:} \emph{i)} The hypersurface $Y$ is embedded in $M$ in such a way that the boundary faces
of $Y$ are in bijective correspondence with the boundary faces of $M$. This means the map $\F(M) \ni F \mapsto F \cap Y \in \F(Y)$
should be a bijection, where $\F(M), \F(Y)$ denotes the boundary faces of $M$ and $Y$ respectively. 
\emph{ii)} Secondly, we assume that for the given Lie structure $2\V$ there is an integrating Lie groupoid $\G$ which is Hausdorff, amenable and has the local triviality property: $\G_F \cong F \times F \times G$ for any open face $F$ of $M$ (where $G$ is an isotropy Lie group).
Also $\G_{M_0} \cong M_0 \times M_0$ is the pair groupoid on the interior and $\A_{|M_0} \cong TM_0$ the tangent bundle on the interior. 

Note that the first assumption poses no loss of generality in the consideration of boundary value problems: We can always
consider a small tubular neighborhood of $Y$ and restrict the groupoid accordingly. Away from such a neighborhood there are no boundary value problems to be considered.
The second condition is needed in the parametrix construction at the end of this work. In the main body of the paper we give examples of Lie structures for which it is known that they have integrating
Lie groupoids that are Hausdorff, amenable and locally trivial. From now on we fix a Lie manifold with boundary with the notation
given above and fulfilling the previous assumptions.

The goal is to construct a Boutet de Monvel calculus for general pseudodifferential boundary value problems adapted to this data. 
We model our construction on the pseudodifferential calculus on a Lie manifold as described in \cite{ALN}. 
The authors define a representation of pseudodifferential operators on a Lie groupoid, i.e. a $\ast$-homomorphism
from equivariant operators acting on the groupoid fibers to operators acting on the original manifold with corners. 
Since the pseudodifferential calculus on a Lie groupoid is closed under composition (cf. \cite{NWX}), a representation theorem automatically yields closedness under composition
for the corresponding pseudodifferential calculus on the given Lie manifold.
In order to construct a groupoid calculus for boundary value problems we first consider the Lie algebroid $A_{2\V} \to M$ such that $\Gamma(\A_{2\V}) \cong 2\V$ using
the Serre Swan theorem. 
We fix a Lie groupoid $\G \rightrightarrows M$ fulfilling our assumptions such that $\A(\G) \cong \A_{2\V}$. 
On the boundary Lie structure on $Y$ we also obtain a groupoid $\G_{\partial} = \G_Y^Y \rightrightarrows Y$ with 
an associated Lie algebroid $\pi_{\partial} \colon \A_{\partial} \to Y$.
We define a \emph{boundary structure} in the following way: We introduce $\X = \G^Y = r^{-1}(Y)$, which is
a longitudinally smooth space and can be viewed as a desingularization of $Y \times M$ with regard to the diagonal $\Delta_Y$, suitably
embedded as manifolds with corners, as well as $\Xop = \G_Y = s^{-1}(Y)$ of $M \times Y$ with regard to $\Delta_Y$.
There is a canonical diffeomorphism $\flip \colon \X \to \Xop$. 
Additionally, the spaces $\X$ and $\Xop$ implement a bibundle correspondence between $\G$ and $\G_{\partial}$ as defined by Hilsum and Skandalis, \cite{HS2}.   
Since the hypersurface $Y$ divides the double $M = 2X$ we denote by $X := X_{+}$ the \emph{right} half and by $X_{-}$ the \emph{left} half.
These halves have corresponding Lie structures and hence corresponding groupoids $\G^{\pm} \rightrightarrows X_{\pm}$.  
On the symbols of pseudodifferential operators from the groupoid calculus we impose a fiberwise transmission property 
with regard to the subgroupoids $\G^{+}, \G^{-} \subset \G$. The compatibility requirements we will state particularly imply that $\G^{+}, \G^{-}$ have fiberwise boundaries consisting of the fibers
$\X_{x}$ for $x \in X_{\pm}$. The Boutet de Monvel operators are defined and adapted to data given by the boundary structure, i.e. the tuple $(\G, \G_{\partial}, \G^{\pm}, \X, \Xop, \flip)$.  
The boundary structure depends on the initial Lie structure and integrability properties of the corresponding Lie algebroids.

\subsection{Results}

The Boutet de Monvel calculus adapted to this data is denoted by $\B^{0,0}(\G^{+}, \G_{\partial})$ (of order $0$ and type $0$) and
consists of matrices of operators, which are equivariant families of operators, such that
\[
\B^{0,0}(\G^{+}, \G_{\partial}) \subset \End \begin{pmatrix} C_c^{\infty}(\G^{+}) \\ \oplus \\ C_c^{\infty}(\G_{\partial}) \end{pmatrix}.
\]

The first objective of this work is the proof of the following result.

\begin{Thm}
Given a Lie manifold $(X, \V)$ with embedded hypersurface $Y \subset X$ yielding a Lie manifold $X$ with boundary $Y$ such that $M = 2X$, the double.
Then for a pair of associated groupoids $\G \rightrightarrows M, \ \G_{\partial} \rightrightarrows Y$ adapted to a boundary 
structure the equivariant transmission Boutet de Monvel calculus is closed under composition. 
This means that given the order $m \in \Zz$ we have
\[
\B^{m, 0}(\G^{+}, \G_{\partial}) \cdot \B^{0, 0}(\G^{+}, \G_{\partial}) \subseteq \B^{m, 0}(\G^{+}, \G_{\partial}).
\]
\label{Thm:BM}
\end{Thm}

In the next step we describe a \emph{vector-representation} of our algebra. 
Just as in the case of a pseudodifferential operator on a groupoid there is a homomorphism which maps
$\B^{0,0}(\G^{+}, \G_{\partial})$ to an algebra $\B_{\V}^{0,0}(X, Y)$. The first algebra on the left consists of equivariant families on a suitable boundary structure.
This algebra on the right hand side is defined to consist of matrices of pseudodifferential, trace, potential and singular Green operators.
These operators are extensions from the usual operator calculus on the interior manifold with boundary $(X_0, Y_0)$. 
Hence we want to define a homomorphism $\varrho_{BM}$ of algebras from
\[
\End\begin{pmatrix} C_c^{\infty}(\G^{+}) \\ \oplus \\ C_c^{\infty}(\G_{\partial}) \end{pmatrix} \supset \B^{0,0}(\G^{+}, \G_{\partial}) \to \B_{\V}^{0,0}(X, Y) \subset \End\begin{pmatrix} C^{\infty}(X) \\ \oplus \\ C^{\infty}(Y) \end{pmatrix}.
\]

The vector representation is characterized by the \emph{defining property} 
\[
\left(\varrho_{BM}(A) \circ \begin{pmatrix} r \\ r_{\partial} \end{pmatrix} \right) \begin{pmatrix} f \\ g \end{pmatrix} = A\begin{pmatrix} f \circ r \\ g \circ r_{\partial} \end{pmatrix}, \ f \in C^{\infty}(X), \ g \in C^{\infty}(Y), \ A \in \B^{0,0}(\G^{+}, \G_{\partial}) 
\]

where $r \colon \G^{+} \to M, \ r_{\partial} \colon \G_{\partial} \to Y$ denote the corresponding range maps (local diffeomorphisms).

It is a non-trivial task to prove that in certain particular cases $\varrho_{BM}$ furnishes an isomorphism between these two algebras.
Furthermore, as can already be shown by simply viewing the special case of pseudodifferential operators, it is not true in general.
Instead we prove an analog of a result due to Ammann, Lauter and Nistor (cf. \cite{ALN}). 

\begin{Thm}
Given the vector representation $\varrho_{BM}$ and a boundary structure we have the following isomorphism
\[
\varrho_{BM} \left(\B^{m,0}(\G^{+}, \G_{\partial})\right) \cong \B_{\V}^{m,0}(X, Y).
\]
\label{Thm:BM2}
\end{Thm}

A priori, the inverse of an invertible Boutet de Monvel operator will not be contained in our calculus due to the definition via compactly supported distributional kernels. 
We define a completion $\overline{\B}_{\V}^{-\infty,0}(X, Y)$ of the residual Boutet de Monvel operators with regard to the family of norms
of operators $\L\left(\begin{matrix} H_{\V}^t(X) \\ \oplus \\ H_{\W}^t(Y) \end{matrix}, \ \begin{matrix} H_{\V}^r(X) \\ \oplus \\ H_{\W}^r(Y) \end{matrix}\right)$ on Sobolev spaces, cf. \cite{ALNV}.
Define the completed algebra of Boutet de Monvel operators as
\[
\overline{\B}_{\V}^{0,0}(X, Y) = \B_{\V}^{0,0}(X, Y) + \overline{\B}_{\V}^{-\infty, 0}(X, Y).
\]
The resulting algebra contains inverses and has favorable algebraic properties, e.g. it is spectrally invariant, cf. section \ref{properties}. 
We obtain a parametrix construction after defining a notion of \emph{Shapiro-Lopatinski ellipticity}.
The indicial symbol $\R_{F}$ of an operator $A$ on $X$ is an operator $\R_{F}(A)$ defined as the restriction to a singular hyperface $F \subset X$ (see \cite{ALN}). 
Note that if $F$ intersects the boundary $Y$ non-trivially we obtain in this way a non-trivial Boutet de Monvel
operator $\R_F(A)$ defined on the Lie manifold $F$ with boundary $F \cap Y$. 


We say that an operator $A \in \overline{\B}_{\V}^{0,0}(X, Y)$ is \emph{$\V$-elliptic} if the principal symbol and the principal boundary symbol of
$A$ are both pointwise invertible.
If a $\V$-elliptic operator $A$ has pointwise invertible indicial symbols $\R_{F}(A)$ for each singular hyperface $F$
we call $A$ \emph{elliptic}. Then we prove the following result.

\begin{Thm}
\emph{i)} Let $A \in \overline{\B}_{\V}^{0,0}(X, Y)$ be $\V$-elliptic. There is a parametrix $B \in \overline{\B}_{\V}^{0,0}(X, Y)$ of $A$, in the sense
\[
I - AB \in \overline{\B}_{\V}^{-\infty, 0}(X, Y), \ I - BA \in \overline{\B}_{\V}^{-\infty, 0}(X, Y). 
\]


\emph{ii)} Let $A \in \overline{\B}_{\V}^{0,0}(X, Y)$ be elliptic. There is a parametrix $B \in \overline{\B}_{\V}^{0,0}(X, Y)$ of $A$ up to 
compact operators
\[
I - AB \in \K\begin{pmatrix} L_{\V}^2(X) \\ \oplus \\ L_{\W}^2(Y) \end{pmatrix}, \ I - BA \in \K\begin{pmatrix} L_{\V}^2(X) \\ \oplus \\ L_{\W}^2(Y) \end{pmatrix}. 
\]

\label{Thm:parametrix}
\end{Thm}

An application of groupoids to boundary value problems can be found in \cite{CY}. We also refer to recent work of Debord and
Skandalis where the Boutet de Monvel algebra is being studied using deformation groupoids \cite{DS}. 
It would be interesting to explore applications of the calculus to cases of manifolds with piecewise smooth boundaries previously considered in the literature.
One example are Lipschitz boundaries, cf. \cite{JK}, \cite{MT}, \cite{MMT}. 
On the other hand potential applications to the $\overline{\partial}$-problem could be explored, see \cite{F}.

The paper is organized as follows. In Section 2 we recall the definition of Boutet de Monvel's calculus in the standard case and 
study an example for our construction. 
In Section 3 we briefly summarize the definitions and notation for Lie groupoids, groupoid actions and Lie algebroids.
There we also introduce operators defined via their Schwartz kernels and discuss reduced kernels.
Section 4 is concerned with the notion of boundary structure. We motivate the definition by considering known examples of Lie structures with corresponding integrating Lie groupoids.
We prove that under our assumptions such a boundary structure or tuple exists.
In Section 5 we introduce the extended operators of Boutet de Monvel type which are special instances of the operators defined in Section 3.
Then we show how to compose these operators in Section 6. 
Section 7 is concerned with the definition of the Boutet de Monvel calculus with regard to a given boundary structure. We prove the closedness under composition and the representation theorem.
In Section 8 we describe several properties of the calculus, including order reductions and continuity on Sobolev spaces.
Section 9 is concerned with the construction of parametrices in the Lie calculus. 

\subsection*{Acknowledgements}

First and foremost I thank my advisor Elmar Schrohe for his invaluable support during my work on this project.
Part of the research was completed while the author was visiting the University of Metz Lorraine, UFR Math\'ematiques.
I'm grateful to Victor Nistor for his friendly hospitality and many helpful discussions.
I thank Georges Skandalis for helpful discussions during my visit to Universit\'e Paris Diderot.
Additionally, Magnus Goffeng and Julie Rowlett made several suggestions which led to an improvement of the exposition.
This research is for the most part based on by my PhD thesis and was conducted while I was a member of the Graduiertenkolleg GRK 1463 at Leibniz University of Hannover.
I thank the Deutsche Forschungsgemeinschaft (DFG) for their financial support.

\section{Boutet de Monvel's calculus}

\subsection{Properties}

The calculus of Boutet de Monvel contains the classical boundary value problems as well as their inverse if it exists and parametrices.
This calculus was developed in 1971, see \cite{BM}.
Let for the moment $X$ be a smooth compact manifold with boundary denoted $\partial X = Y$.
Let $P$ be a pseudodifferential operator on a smooth neighborhood $M$ of the manifold with boundary.
We denote by $P_{+} = \chi^{+} P \chi^{0}$ the operator obtained from $P$ via extension by zero $\chi^{0}$ of functions defined on $X$ to functions on $M$ and restriction $\chi^{+}$ to $X$. 
The pseudodifferential operator $P$ should fulfill the transmission property, i.e. $P_{+}$ should map functions smooth up to the boundary to functions again smooth up to the boundary (\cite{G}, p. 23, (1.2.6)).
Furthermore, $G \colon C^{\infty}(X) \to C^{\infty}(X)$ is a singular Green operator (\cite{G}, p. 30), $K \colon C^{\infty}(\partial X) \to C^{\infty}(X)$ is a potential operator (\cite{G}, p. 29) and $T \colon C^{\infty}(X) \to C^{\infty}(\partial X)$ is a trace operator (\cite{G}, p. 27).
Additionally, $S$ is a pseudodifferential operator on the boundary.
Elements of the calculus consist of matrices of operators
\begin{align}
A &= \begin{pmatrix} P_{+} + G & K \\
T & S \end{pmatrix} \colon \begin{matrix} C^{\infty}(X, E_1) \\ \oplus \\ C^{\infty}(\partial X, F_1) \end{matrix} \to \begin{matrix} C^{\infty}(X, E_2) \\ \oplus \\ C^{\infty}(\partial X, F_2) \end{matrix} \in \B^{m,d}(X, \partial X) \label{BM}
\end{align}

The calculus has the following notable properties.
\begin{itemize}
\item If the bundles match, i.e. $E_1 = E_2 = E, \ F_1 = F_2 = F$ the calculus is closed under composition. 

\item If $F_1 = 0, \ G = 0$ and $K, \ S$ are not present, the classical BVP's are contained in the calculus, e.g. Dirichlet problem.

\item If $F_2 = 0$ and $T, \ S$ are not present, it contains inverses of classical BVP's if they exist. 
\end{itemize}

It is non-trivial to establish the closedness under composition, which means that two Boutet de Monvel operators composed
are again of this type, compare the standard references \cite{G}, \cite{RS}. 

\subsection{An elliptic boundary value problem}

In this section we will study a Shapiro Lopatinski elliptic boundary value problem adapted to a Lie manifold with boundary.
Specifically, we consider the appropriate generalization of the problem discussed in Section 2 of \cite{K}. 
We therefore fix the data in the introduction, a Lie manifold $(X, \V)$ with boundary $(Y, \W)$ and double $(M, 2\V)$. 

Let smooth hermitian vector bundles $\tilde{E}, \tilde{F} \to M$, restrictions $E, F \to X$ as well as $F_j \to M$ for $j = 1, \cdots, L$ be given.
We fix the differential operators on the double Lie manifold $B_j \in \Diff_{2 \V}^{m_i}(M; \tilde{E}, \tilde{F})$ for $j = 1, \cdots, L$ and $P \in \Diff_{2\V}^m(M, \tilde{E}, \tilde{F})$.
Set $T = (\gamma_Y \circ (B_1)_{+}, \cdots, \gamma_Y \circ (B_L)_{+})^t$ for the \emph{trace operator} and let $d := \max_{j=1}^L m_j + 1$. 

We consider the following boundary value problem 
\begin{align*}
P_{+} u = f \ \text{in} \ X, \\
T u = g \ \text{on} \ Y.
\end{align*}

The problem $\A = \begin{pmatrix} P_{+} \\ T \end{pmatrix}$ in particular yields a continuous linear operator on the appropriate Sobolev spaces
\[
\A = \begin{pmatrix} P_{+} \\ T \end{pmatrix} \colon H_{\V}^{s}(X, E) \to \begin{matrix} H_{\V}^{s-m}(X, F) \\ \oplus \\ \oplus_{j=1}^L H_{\W}^{s - m_j - \frac{1}{2}}(Y, F_{j|Y}) \end{matrix}.
\]

\begin{Thm}
Assume the boundary value problem $\A = \begin{pmatrix} P_{+} \\ T \end{pmatrix}$ is Shapiro-Lopatinski elliptic (in the sense of Definition \ref{Def:SL} in section \ref{parametrices}),
then for $s > \max\{m, d\} - \frac{1}{2}$ 
\[
\A = \begin{pmatrix} P_{+} \\ T \end{pmatrix} \colon H_{\V}^{s}(X, E) \to \begin{matrix} H_{\V}^{s-m}(X, F) \\ \oplus \\ \oplus_{j=1}^L H_{\W}^{s - m_j - \frac{1}{2}}(Y, F_{j|Y}) \end{matrix}.
\]

is Fredholm.

There is a parametrix $\P = \begin{pmatrix} Q_{+} + G & K_1 & \cdots & K_L \end{pmatrix}$ of $\A$ in the completed Lie calculus

$\overline{\B}_{\V}^{-m, \max\{0, d -m\}}(X, Y)$ (cf. section \ref{properties}).

Here $G$ is a singular Green operator of order $-m$ and $K_j$ are potential operators of order $-m_j - \frac{1}{2}$. 
\label{Thm:exa}
\end{Thm}

\begin{proof}
We fix the pseudodifferential order reductions (cf. section \ref{properties}, \cite{ALNV}, \cite{N})
\[
R_j \colon H_{\W}^s(Y, F_{j|Y}) \iso H_{\W}^{s - \mu_j}(Y, F_{j|Y})
\]

in the completed pseudodifferential calculus $\overline{\Psi}_{\W}^{m_j}(Y)$, $\mu_j = m - m_j - \frac{1}{2}, \ j = 1, \cdots, L$. 

Set
\[
R := \begin{pmatrix} R_1 & 0 & 0 \\
\vdots & \ddots & \vdots \\
0 & 0 & R_L \end{pmatrix}
\]

and consider for $s > \max\{m,d\} - \frac{1}{2}$ the operator
\[
\B = \begin{pmatrix} 1 & 0 \\
0 & R \end{pmatrix} \begin{pmatrix} P_{+} \\ T \end{pmatrix} \colon H_{\V}^s(X, E) \to \begin{matrix} H_{\V}^{s-m}(X, F) \\ \oplus \\ H_{\W}^{s-m}(Y, J_{+}) \end{matrix}
\]

where $J_{+} = \oplus_{j=1}^L F_{j|Y}$. 

We obtain an elliptic pseudodifferential boundary value problem contained in the completed calculus $\B \in \Bc_{\V}^{m,d}(X, Y)$. 
The ellipticity follows by the multiplicativity of the principal and principal boundary symbol in Lemma \ref{Lem:mult}. 
By Theorem \ref{Thm:parametrix} there is a parametrix $\tilde{\P} \in \Bc_{\V}^{-m, (d - m)_{+}}(X, Y)$ 
up to residual terms. 
Hence $\P = \tilde{\P} \begin{pmatrix} 1 & 0 \\
0 & R \end{pmatrix}$ yields a parametrix of $\A$.
The Fredholm property of $\P$ follows by another application of Theorem \ref{Thm:parametrix}, part (b). 
\end{proof}

\section{Groupoids, actions and algebroids}

\subsection{Lie groupoids}


\begin{Def}
Groupoids are small categories in which every morphism is invertible. 
\label{Def:grpd}
\end{Def}

First we will introduce and fix the notation for the rest of this paper. 
Then we will give the definition of a Lie groupoid. 
For more details on groupoids we refer the reader for example to the book \cite{PAT}.

\begin{Not}
A groupoid will be denoted $\G \rightrightarrows \Gop$. We denote by $\Gmor$ the set of morphisms and by $\Gop$ the set
of objects. By a common abuse of notation we write $\G$ for $\Gmor$. We have the range / source maps $r, \ s \colon \G \to \Gop$ such that $\gamma \in \G$ is 
\[
\gamma \colon s(\gamma) \to r(\gamma).
\]

The set of composable arrows is given as pullback
\[
\Gpull := \Gmor \times_{\Gop} \Gmor = \{(\gamma, \eta) \in \G \times \G : r(\eta) = s(\gamma)\}.
\]

Also denote the inversion
\[
i \colon \G \to \G, \ \gamma \mapsto \gamma^{-1}
\]

and unit map
\[
u \colon \Gop \to \G, \ x \mapsto u(x) = \id_x \in \G.
\]

Multiplication is denoted by
\[
m \colon \Gpull \to \G, \ (\gamma, \eta) \mapsto \gamma \cdot \eta.
\]

We also set 
\[
\G_x := s^{-1}(x), \ \G^x := r^{-1}(x), \ \G_x^x := \G_x \cap \G^x
\]

for the $r$ and $s$ fibers and their intersection $\G_x^x$. The latter is easily checked to be a group for each $x \in \Gop$.

\label{Not:grpd}
\end{Not}

\textbf{Axioms:} One can summarize the maps in a sequence 
\[
\xymatrix{
\Gpull \ar@{->>}[r]^{m} & \G \ar@{>->>}[r]^{i} & \G \ar@{->>}@< 2pt>[r]^-{r,s} \ar@{->>}@<-2pt>[r]^{} & \Gop \ar@{>->}[r]^{u} & \G.
}
\]

With the above notation we can give an alternative way of defining groupoids axiomatically as follows. 

\emph{(i)} $(s \circ u)_{|\Gop} = (r \circ u)_{|\Gop} = \id_{\Gop}$.

\emph{(ii)} For each $\gamma \in \G$
\[
(u \circ r)(\gamma) \cdot \gamma = \gamma, \ \gamma \cdot (u \circ s)(\gamma) = \gamma.
\]

\emph{(iii)} $s \circ i = r, \ r \circ i = s$. 

\emph{(iv)} For $(\gamma, \eta) \in \Gpull$ we have
\[
r(\gamma \cdot \eta) = r(\gamma), \ s(\gamma \cdot \eta) = s(\eta).
\]

\emph{(v)} For $(\gamma_1, \gamma_2), \ (\gamma_2, \gamma_3) \in \Gpull$ we have
\[
(\gamma_1 \cdot \gamma_2) \cdot \gamma_3 = \gamma_1 \cdot (\gamma_2 \cdot \gamma_3).
\]

\emph{(vi)} For each $\gamma \in \G$ we have
\[
\gamma^{-1} \cdot \gamma = \id_{s(\gamma)}, \ \gamma \cdot \gamma^{-1} = \id_{r(\gamma)}. 
\]

\begin{Def}
The $7$-tuple $(\Gop, \Gmor, r,s,m,u,i)$ defines a \emph{Lie groupoid} if and only if $\G \rightrightarrows \Gop$ is a groupoid, $M := \Gop, \ \Gmor$ are $C^{\infty}$-manifolds (with corners),
all the maps are $C^{\infty}$ and $s$ is a submersion.
\label{Def:Liegrpd}
\end{Def}

\begin{Rem}
We notice that $r$ is automatically a submersion due to the axiom \emph{iii)}. 
Hence the pullback
\[
\xymatrix{
\Gpull \ar[d]_{p_2} \ar[r]^{p_1} & \G \ar[d]_{s} \\
\G \ar[r]^{r} & M
}
\]

exists in the $C^{\infty}$-category if $\G$ is $C^{\infty}$ and thus $\Gpull$ is a smooth manifold as well.
\label{Rem:Liegrpd}
\end{Rem}

\subsection{Groupoid actions}


Given a Lie groupoid $\G \rightrightarrows M$ we introduce spaces $\X$ which are fibred over $\Gop$ and such that
$\G$ acts on $\X$. 
This notion as well as some of the notation is adapted from the paper \cite{PAT2}.

\begin{Def}
Let $(\X, q)$ be a $\G$-space, i.e. $q \colon \X \to M$ is a smooth map and $\X$ is a smooth manifold.
Set $\X \ast \G := \X \times_M \G = \{(z, \gamma) \in \X \times \G : q(z) = r(\gamma)\}$ to be the \emph{composable elements}.
We say that $\G$ \emph{acts on} $\X$ from the right if the following conditions hold:

\emph{i)} For each $(z, \gamma) \in \X \ast \G$
\[
q(z \cdot \gamma) = s(\gamma).
\]

\emph{ii)} For each $(z, \gamma) \in \X \ast \G$ and $(\gamma, \eta) \in \Gpull$
\[
z \cdot (\gamma \cdot \eta) = (z \cdot \gamma) \cdot \eta.
\]

\emph{iii)} For each $(z, \gamma) \in \X \ast \G$ we have
\[
(z \cdot \gamma) \cdot \gamma^{-1} = z.
\]

A left action of $\G$ on a $\G$-space $(\X, p)$ is a right-action in the opposite category $\G^{\mathrm{op}}$. 
\label{Def:actions}
\end{Def}
%
%
%
%
%
%
%

\begin{Rem}
Note that for any $\H$-space $(\X, p)$ which is additionally \emph{fibered} meaning $p$ is a surjective submersion then the pullback $\H \ast \X$ exists in the $C^{\infty}$-category if $\H, \ \X$ are $C^{\infty}$. Analogously for a fibered $\G$-space.
\label{Rem:actions}
\end{Rem}

Consider the following actions of two Lie groupoids $\G, \H$:
\[
\begin{tikzcd}[every label/.append style={swap}]
\H \arrow[symbol=\circlearrowleft]{r} & \arrow{dl}{p} \X \arrow[symbol=\circlearrowright]{r} \arrow{dr}{q} & \G \\
\Hop & & \Gop 
\end{tikzcd}
\]


We can define a so-called left Haar system on $\X$ induced by the action of $\H$ 
and analogously a right Haar system induced by the action of $\G$. 
This enables us to define left- and right-operators coming from the actions.

Let $\{\lambda_x\}_{x \in \Gop}$ be a Haar system induced on $\X$ by the right action of $\G$, see also \cite{PAT2}, p. 6. 
This is a family of measures such that
\begin{itemize}
\item The support is $\supp \ \lambda_x = \X_x$ for each $x \in \Gop$. 

\item The map
\[
\Gop \ni x \mapsto \int_{\X_x} f \,d\lambda_x 
\]

is $C^{\infty}$.

 \item We have the invariance condition
\begin{align}
& \int_{\X_{r(\gamma)}} f(z \cdot \gamma) \,d\lambda_{r(\gamma)}(z) = \int_{\X_{s(\gamma)}} f(w) \,d\lambda_{s(\gamma)}(w). \label{Hinv}
\end{align}
\end{itemize}


Fix the \emph{right-multiplication} for given $\gamma \in \G$ 
\[
r_{\gamma} \colon \X_{r(\gamma)} \to \X_{s(\gamma)}, \ z \mapsto z \cdot \gamma.
\]

This is a diffeomorphism. 

The induced operators acting on $C^{\infty}$-functions are given by 
\[
R_{\gamma} \colon C_c^{\infty}(\X_{s(\gamma)}) \to C_c^{\infty}(\X_{r(\gamma)}), \ (R_{\gamma} f)(z) := f(z \cdot \gamma), \ z \in \X.
\]

These operators $\R_{\gamma}$ yield $\ast$-(anti-)homomorphisms since $(R_{\gamma})^{-1} = R_{\gamma^{-1}}$ is the inverse and $R_{\gamma \cdot \eta} = R_{\eta} \circ R_{\gamma}, \ (\gamma, \eta) \in \Gpull$.


\begin{Def}
\emph{i)} A continuous linear operator $T \colon C_c^{\infty}(\G) \to C_c^{\infty}(\X)$ is called a \emph{right $\X$-operator} if and only if
$T = (T_x)_{x \in \Gop}$ is a family of continuous linear operators $T_x \colon C_c^{\infty}(\G_x) \to C_c^{\infty}(\X_x)$ such that
\begin{align}
R_{\gamma^{-1}} T_{r(\gamma)} R_{\gamma} = T_{s(\gamma)}, \ \gamma \in \G \label{inv}.
\end{align}

This can be expressed alternatively by requiring the following diagram to commute for each $\gamma \in \G$ 
\[
\xymatrix{ 
C_c^{\infty}(\G_{s(\gamma)}) \ar[r]^{T_{s(\gamma)}} & C_c^{\infty}(\X_{s(\gamma)}) \ar[d]_{R_{\gamma}} \\
C_c^{\infty}(\G_{r(\gamma)}) \ar[r]^{T_{r(\gamma)}} \ar[u]_{R_{\gamma^{-1}}} & C_c^{\infty}(\X_{r(\gamma)}).
}
\]

\emph{ii)} By analogy $\tilde{T} \colon C_c^{\infty}(\X) \to C_c^{\infty}(\H)$ is a \emph{left $\X$-operator} if and only if $\tilde{T} = (\tilde{T}^y)_{y \in \Hop}$ is a family 
of continuous linear operators $\tilde{T}^y \colon C_c^{\infty}(\X^y) \to C_c^{\infty}(\H^y)$ such that the diagram
\[
\xymatrix{ 
C_c^{\infty}(\X^{s(\gamma)}) \ar[r]^{\tilde{T}^{s(\gamma)}} & C_c^{\infty}(\H_{s(\gamma)}) \ar[d]_{L_{\gamma^{-1}}} \\
C_c^{\infty}(\X^{r(\gamma)}) \ar[u]_{L_{\gamma}} \ar[r]^{\tilde{T}^{r(\gamma)}} & C_c^{\infty}(\H_{r(\gamma)})
}
\]

commutes for each $\gamma \in \H$ where $L_{\gamma}$ denotes in this case the corresponding left multiplication. 

\label{Xop}
\end{Def}

The next Proposition tells us that the family of Schwartz kernels $(k_x)_{x \in \Gop}$ for a given $\X$-operator can be replaced by a 
so-called reduced kernel. 
This is not unlike the situation for groupoids and the pseudodifferential calculus where reduced kernels are used extensively (cf. \cite{NWX}).

\begin{Prop}
Given a right-$\X$-operator $T \colon C_c^{\infty}(\G) \to C_c^{\infty}(\X)$. 
Then for $u \in C_c^{\infty}(\G), \ z \in \X$ 
\[
(Tu)(z) = \int_{\G_{q(z)}} k_T(z \cdot \gamma^{-1}) u(\gamma) \,d\mu_{q(z)}(\gamma)
\]

with $k_T(z \cdot \gamma^{-1}) := k_{r(\gamma)}(z, \gamma)$ depending only on $z \cdot \gamma^{-1} \in \X$ for each $(z, \gamma^{-1}) \in \X \ast \G$. 
\label{Prop:reduced}
\end{Prop}

\begin{proof}
First we can write for $z \in \X$
\begin{align*}
& (R_{\gamma^{-1}} T_{r(\gamma)} R_{\gamma})u(z) = (T_{r(\gamma)} R_{\gamma} u)(z \cdot \gamma^{-1}) = \int_{\G_{r(\gamma)}} k_{r(\gamma)}(z \cdot \gamma^{-1}, \eta) u(\eta \gamma) \,d\mu_{r(\gamma)}(\eta) \\
&= \int_{\G_{s(\gamma)}} k_{r(\gamma)}(z \cdot \gamma^{-1}, \tilde{\eta} \cdot \gamma^{-1}) u(\tilde{\eta}) \,d\mu_{s(\gamma)}(\tilde{\eta}) 
\end{align*}

via the substitution $\tilde{\eta} := \eta \gamma$ and invariance of Haar system.
By use of \eqref{inv} we see that the last integral equals
\[
(T_{s(\gamma)} u)(z) = \int_{\G_{s(\gamma)}} k_{s(\gamma)}(z, \eta) u(\eta)\,d\mu_{s(\gamma)}(\eta).
\]

This implies the following identity by the uniqueness of the Schwartz kernel for $T_x$ for each $x \in M$
\begin{align}
\forall_{\gamma \in \G} \ k_{s(\gamma)}(z, \tilde{\eta}) = k_{r(\gamma)}(z \cdot \gamma^{-1}, \tilde{\eta} \cdot \gamma^{-1}). \tag{$*$} \label{*}
\end{align}

To see that $k_T$ is well-defined assume $\beta = z \cdot \gamma^{-1} = \tilde{z} \cdot \tilde{\gamma}^{-1}$ and $\delta = \gamma^{-1} \cdot \tilde{\gamma}$,
then
\begin{align*}
k_{s(\tilde{\gamma})}(\tilde{z}, \tilde{\gamma}) &= k_{s(\delta)}(\tilde{z}, \tilde{\gamma}) =_{\eqref{*}} k_{r(\delta)}(\tilde{z} \delta^{-1}, \tilde{\gamma} \delta^{-1}) \\
&= k_{s(\gamma)}(\beta \tilde{\gamma} \delta^{-1}, \tilde{\gamma} \delta^{-1}) \\
&= k_{s(\gamma)}(z, \gamma). 
\end{align*}

This completes the proof.
\end{proof}

% ...

\subsection{Lie algebroids}

The aim of this section is to give a definition of Lie algebroids and subalgebroids.
We restrict ourselves to the bare minimum needed in the following text of the paper.
For a more detailed exposition the reader may consult e.g. \cite{MOER}. 

\begin{Def}
\begin{itemize}
\item A \emph{Lie algebroid} is a tuple $(E, \ \varrho)$. 
Here $\pi \colon E \to M$ is a vector bundle over a manifold $M$ and $\varrho \colon E \to TM$ is a vector bundle map such that
\[
\varrho \circ [V, W]_{\Gamma(E)} = [\varrho \circ V, \varrho \circ W]_{\Gamma(TM)}
\]

and
\[
[V, fW]_{\Gamma(E)} = f [V,W]_{\Gamma(E)} + \varrho(V)(f) W, \ f \in C^{\infty}(M), \ V, W \in \Gamma(E).
\]

\item Given two Lie algebroids $(\A, \ \varrho)$ and $(\tilde{\A}, \ \tilde{\varrho})$ over the same manifold $M$. 
Then a \emph{Lie algebroid morphism} is a map $\varphi \colon \A \to \tilde{\A}$ making the following diagram commute
\[
\xymatrix{
& TM & \\
\A \ar[dr]_{\pi} \ar[ur]^{\varrho} \ar[r(1.8)]^{\varphi} & & \ar[dl]_{\tilde{\pi}} \ar[ul]^{\tilde{\varrho}} \tilde{\A}  \\
& M &
}
\]

and such that $\varphi$ preserves Lie bracket: $\varphi [V, W]_{\Gamma(\A)} = [\varrho(V), \varrho(W)]_{\Gamma(TM)}$. 
\end{itemize}

\label{Def:algbroid}
\end{Def}



We briefly summarize some relevant facts about the construction of Lie algebroids.

\begin{itemize}
\item For any given Lie groupoid $\G$ we obtain an associated algebroid $\A(\G)$ in a covariantly functorial way.
Define $T^s \G := \ker(ds)$ the \emph{$s$-vertical tangent bundle} as a sub-bundle of $T\G$. 
Denote by $\Gamma(T^s \G)$ the smooth sections and define $\Gamma_R(T^s \G)$ as the sections $V$ such that
\[
V(\eta \gamma) = (R_{\gamma})_{\ast} V_{\eta} \ \text{for} \ (\eta, \gamma) \in \Gpull.
\]

We then define the Lie-algebroid associated to $\G$ via the pullback
\[
\xymatrix{
\A(\G) \ar[d] \ar[r]^{u^{\ast}} & T^s \G \ar[d]_{\pi_{|s}} \\
M \ar[r]^{u} & \G. 
}
\]

In other words $\A(\G) := \{(V, x) | ds(V) = 0, \ u(x) = 1_x = \pi(V)\}$. 

\item There is a canonical isomorphism of Lie-algebras $\Gamma_R(T^s \G) \cong \Gamma(\A(\G))$. The set of smooth sections $\Gamma(\A(\G))$ is a $C^{\infty}(M)$-module with the module operation $f \cdot V = (f \circ r) \cdot V$
with $f \in C^{\infty}(M)$.

\item Let $\A(\G)$ be given as above and define $\varrho \colon \A(\G) \to TM$ by $\varrho := dr \circ u^{\ast}$.
Then $(\A(\G), \varrho)$ so defined furnishes a Lie algebroid.

\item A Lie algebroid is said to be \emph{integrable} if we can find an associated (with connected $s$-fibers) Lie groupoid.
Not every Lie algebroid is integrable \cite{CF}.

\end{itemize}

\section{Boundary structure}
\label{BdyStruct}

With the given Lie manifold with boundary we want to associate a so-called \emph{boundary structure}.
We define a boundary structure and we show that for particular examples of Lie structures there is a boundary structure.
What is necessary in the general case is a certain assumption on the groupoids $\G,  \ \G_{\partial}$, namely they ought to define a bibundle structure
which we are going to specify. The boundary structure is in fact a good analogy for blow-ups of the corners which are the intersections of $Y$ with the (singular) boundary at infinity of
$M$. These blow-ups are in our setup canonically defined in terms of $\G$ and $\G_{\partial}$, the groupoids integrating $\A$ and $\A_{\partial}$.

Recall that a \emph{bibundle} correspondence between two Lie groupoids $\G$ and $\H$ implemented
by $\X$ is a left $\H$-space $(\X, p)$ and a right $\G$-space $(\X, q)$ with free and proper right action of $\G$ on $\X$ such that there is a homeomorphism $\X / \G \iso \Hop$ induced by $p$ and the actions
of $\G$ and $\H$ commute, see also \cite{HS2}. 


In the following we give the axioms necessary to define a boundary structure.



\begin{Def}
A \emph{boundary structure} is defined as a tuple $(\G, \G_{\partial}, \G^{\pm}, \X, \Xop, \flip)$ consisting of a Lie groupoid
$\G \rightrightarrows M$ and two manifolds (possibly with corners) $\X, \ \Xop$ which are diffeomorphic
via a \emph{flip diffeomorphism} $\flip$ and subgroupoids $\G^{\pm} \rightrightarrows X_{\pm}$ of $\G$.

We impose the following axioms on this data:

\emph{i)} $\A(\G) \cong \A_{2\V}, \ \A(\G_{\partial}) \cong \A_{\partial}$ as well as $\A(\G^{\pm}) \cong \A_{\V}$ as Lie algebroids.


\emph{ii)} $\X$ is a right $\G$- and a left $\G_{\partial}$-space and $\Xop$ is a left $\G$- and a right $\G_{\partial}$-space and these actions implement a bibundle correspondence between $\G$ and $\G_{\partial}$. 
The charge maps of the actions $p \colon \X \to Y, \ q \colon \X \to M$ and $p^t \colon \Xop \to M, \ q^t \colon \Xop \to Y$ are such that $p$ and
$q^t$ are surjective submersions.

\emph{iii)} Restricted to the interior we have 
\begin{align*}
& \X_{|Y_0 \times M_0} = p^{-1}(M_0) \cap q^{-1}(Y_0) \cong Y_0 \times M_0, \\
& \Xop_{|M_0 \times Y_0} = (p^t)^{-1}(M_0) \cap (q^t)^{-1}(Y_0) \cong M_0 \times Y_0.
\end{align*}

\emph{iv)} The fibers of $\G^{\pm}$ are the interiors of smooth manifolds with boundary, namely:
\begin{align*}
& \partial_{reg} \G_x^{+} = \X_x, \ x \in X_{+} \\
& \partial_{reg} \G_x^{-} = \X_x, \ x \in X_{-}.
\end{align*}

\label{Def:bdystruct}
\end{Def}


% ...

\begin{Exa}
\emph{i)} Consider a compact manifold $X$ with boundary $\partial X = Y$ and interior $\mathring{X} := X \setminus Y$. Then we also fix the double $M = 2X$. 
In this (trivial) case the spaces are given by $\X := Y \times M, \ \Xop = M \times Y$ with the flip $\flip(x', y) = (y, x'), \ (x', y) \in Y \times M$.
We have here the pair groupoids $\G = M \times M, \ \G_{\partial} = Y \times Y$ as well as $\G^{+} = \mathring{X}_{+} \times \mathring{X}_{+}, \ \G^{-} = \mathring{X}_{-} \times \mathring{X}_{-}$. 
Then $p, q$ are just the projections $\pi_1 \colon Y \times M \to Y, \ \pi_2 \colon Y \times M \to M$.

\emph{ii)} We consider the Lie structure $\V_b := \{V \in \Gamma^{\infty}(TM) : V \ \text{tangent to} \ F_i, \ 1 \leq i \leq N\}$, cf. \cite{M}. 
The Lie algebroid $\A \to M$ is the $b$-tangent bundle such that $\Gamma(\A) \cong \V_b$. 
Following Monthubert \cite{MONT}, we find a Lie groupoid $\G_b(M)$ integrating $\A$ which is $s$-connected, Hausdorff and amenable: 
We start with the set 
\[
\Gamma_b(M) = \{(x, y, \lambda) \in M \times M \times (\Rr_{+})^N : \rho_i(x) = \lambda_i \rho_i(y), \ 1 \leq i \leq N\}
\]
endowed with the  structure  $(x, y, \lambda) \circ (y, z, \mu) = (x, z, \lambda \cdot \mu), \ (x, y, \lambda)^{-1} = (y,x, \lambda^{-1})$
and $r(x, y, \lambda) = x, \ s(x, y, \lambda) = y,\ u(x) = (x,x,1)$. Here multiplication $\lambda\cdot \mu$ and inversion $\lambda^{-1}$ are componentwise.

We then define the $b$-groupoid $\G_b(M)$ as the $s$-connected component (the union of the connected components of the $s$-fibers of $\Gamma_b(M)$),
i.e. $\G_b(M) := \C_s \Gamma_b(M)$. 

\emph{iii)} Fix the Lie structure $\V_{c_l}$ of 
generalized cusp vector fields for $l \geq 2$ given by the local generators in a tubular neighborhood of a boundary hyperface:
$\{x_1^l \partial_{x_1}, \partial_{x_2}, \cdots, \partial_{x_n}\}$. 
Let us recall the construction of the associated Lie groupoid $\G_l(M)$, the so-called generalized cusp groupoid,  given in \cite{LMN} for the benefit of the reader. 
We set
\[
\Gamma_l(M) := \{(x, y, \mu) \in M \times M \times (\Rr_{+})^N : \mu_i \rho_i(x)^l \rho_i(y)^l = \rho_i(x)^l - \rho_i(y)^l\}
\]
with structure $r(x, y, \lambda) = x, \ s(x, y, \lambda) = y, \ u(x) = (x,x,0)$ and $(x, y, \lambda) (y, z, \mu) = (x, z, \lambda + \mu)$. We then define $\G_l(M)$ as the $s$-connected component of $\Gamma_l(M)$.

\emph{iv)} The following example of the fibered cusp calculus is from Mazzeo and Melrose \cite{MM} and we use the formulation and notation for manifolds with fibered corners as given in \cite{DLR}. 
We briefly recall the definition of the associated groupoid and refer to loc. cit. for the details. See also \cite{Guil} for the 
precise geometric construction of the Lie groupoid for a different type of fibered cusp Lie structure.
Write  $\pi = (\pi_1, \cdots, \pi_N)$, where $\pi_i \colon F_i \to B_i$ are fibrations; $B_i$ is the base, which is a compact manifold with corners. 
Define the Lie structure
\[
\V_{\pi} := \{V \in \V_b : V_{|F_i} \ \text{tangent to the fibers} \ \pi_i \colon F_i \to B_i, \ V \rho_i \in \rho_i^2 C^{\infty}(M)\}.
\]
Then $\V_{\pi}$ is a finitely generated $C^{\infty}(M)$-module and a Lie sub-algebra of $\Gamma^{\infty}(TM)$. 
The corresponding groupoid is amenable  \cite[Lemma 4.6]{DLR}; as a set it is defined as
\[
\G_{\pi}(M) := (M_0 \times M_0) \cup \left(\bigcup_{i = 1}^{N} (F_i \times_{\pi_i} T^{\pi} B_i \times_{\pi_i} F_i) \times \Rr\right), 
\]
where $T^{\pi} B_i$ denotes the algebroid of $B_i$.
\label{Exa:bdystruct}
\end{Exa}

\begin{Thm}
For the Lie structure $\V \in \{\V_b, \V_{c_l}, \V_{\pi}\}, \ l \geq 2$ there is a boundary structure.
\label{Thm:bdystruct}
\end{Thm}

\begin{proof}
The claim is proven for the case $\V = \V_b$. For the other cases the argument goes along the same lines, so we omit it.
We have the fixed boundary defining functions for the hyperfaces of $M$ and denote this family by $(p_j)_{j \in I}$. 
On $Y$ there are the boundary defining functions (relative to $Y$), denoted by $(q_j)_{j \in I}$. 
These are the boundary defining functions of the faces from the intersections of $Y$ with the strata of $M$. 
We set $\X := r^{-1}(Y) = \G(M)_M^Y$ and $\Xop := s^{-1}(Y) = \G(M)_Y^M$. 
Consider now the topology of $\X$ and $\Xop$. It is defined in local charts by the rule
\begin{align*}
& Y_0 \times M_0 \ni (x_n', y_n) \to (x', y, \lambda = (\lambda_j)_{j \in I}) :\Leftrightarrow \frac{q_j(x_n')}{p_{i}(y_n)} \to \lambda_j, \ n \to \infty, \ i, j \in I, \
x_n' \to x', \ y_n \to y.
\end{align*}

The right action of $\G(M)$ on $\X$ is given by right composition and the left action of $\G(Y)$ is given by left composition.
On the interior we only have the pair groupoids. This yields the trivial actions
\[
\begin{tikzcd}[every label/.append style={swap}]
Y_0 \times Y_0 \arrow[symbol=\circlearrowleft]{r} & \arrow{dl}{p = \pi_1} Y_0 \times M_0 \arrow[symbol=\circlearrowright]{r} \arrow{dr}{q = \pi_2} & M_0 \times M_0 \\
Y_0 & & M_0 
\end{tikzcd}
\]
These actions can be extended continuously to the closure of $Y_0 \times M_0$ in $\G$ and also of $M_0 \times Y_0$,
and we obtain from the definition of the topology that
\[
\X = \overline{Y_0 \times M_0}^{\G}, \ \Xop = \overline{M_0 \times Y_0}^{\G}. 
\]

The continued actions are defined
\[
\begin{tikzcd}[every label/.append style={swap}]
\G(Y) \arrow[symbol=\circlearrowleft]{r} & \arrow{dl}{p = r_{\partial}} \X \arrow[symbol=\circlearrowright]{r} \arrow{dr}{q = s} & \G(M) \\
Y & & M 
\end{tikzcd}
\]
Axiom \emph{i)} holds because of \cite{MONT} where it was shown that the given $b$-groupoids integrate the Lie structure of $b$-vector fields.
It is immediate to check that the actions commute and by definition the charge map $q$ induced by the source map of $\G(M)$ is a surjective submersion. 
The action of $\G(M)$ on its space of units is free and proper, hence the action of $\G(M)$ by right composition on $\X$ is free and proper.
Here properness means that the map $\X \ast \G(M) \to \X \times \X, \ (z, \gamma) \mapsto (z \gamma, z)$ is a homeomorphism onto
its image. 
Let $z \gamma = z$ then $q(z) = r(z) = r(\gamma)$ which implies composability and $s(\gamma) = q(z \gamma) = q(z)$.
Hence $\gamma = \id_{q(z)}$ which verifies that the action is free. The same can be proven for the left action under our assumption,
but we do not need this fact. 
Finally, we need to check that we have a diffeomorphism $\X / \G(M) \iso Y$ induced by the charge map $p$.
We first check this for the groupoids $\Gamma(M)$ and $\Gamma(Y)$, then we take the $s$-connected components which proves the assertion
for the groupoid $\G(M)$ and $\G(Y)$.
We have to show that $p(z) = p(w)$ for $z, w \in \X$ if and only if there is a necessarily unique $\eta \in \G(M)$ such that $w = z \cdot \eta$.
Let $z = (x', y, (\lambda_i)_{i \in I}), \ w = (x', \tilde{y}, (\mu_i)_{i \in I})$ and set $\eta = \left(y, \tilde{y}, \left(\frac{\mu_i}{\lambda_i}\right)_{i \in I}\right)$. 
By definition of the topology of the groupoids fix the sequences $(x_n', y_n)$ such that $\frac{q_j(x_n')}{p_j(\tilde{y}_n)} \to \lambda_j, \ j \in I, \ n \to \infty$
and $(x_n', y_n)$ such that $\frac{q_j(x_n')}{p_j(y_n)} \to \mu_j, \ j \in I, \ n \to \infty$. 

Then $\eta \in \Gamma(M)$ since
\[
\frac{p_i(y_n)}{p_i(\tilde{y}_n)} = \frac{q_i(x_n')}{p_i(\tilde{y}_n)} \left(\frac{q_i(x_n')}{p_i(y_n)}\right)^{-1} \to \frac{\mu_i}{\lambda_i}, \ i \in I, \ n \to \infty.
\]

This concludes the proof of the isomorphism $\X / \G(M) \cong Y$. Hence we obtain a bibundle correspondence and Axiom \emph{ii)} is verified.
Part \emph{iii)} follows from the definition of the actions we just gave. 
Also note that the flip diffeomorphism $\flip \colon \X \iso \Xop$ is defined by
\[
\flip \colon (x', y, (\lambda_i)_{i \in J}) \mapsto \left(y, x', \left(\frac{1}{\lambda_i} \right)_{i \in J}\right).
\]


It remains to verify condition \emph{iv)}. For this we define $\G^{\pm} := \G(X_{\pm})$ and prove that this groupoid
has the required property.
Thus we want to show that
\[
\partial_{reg} \G_x^{+} = \X_x, \ x \in X.  
\]

The boundary is possibly empty (for $x$  not incident to the hypersurface $Y$).
We have to distinguish two cases: $x$ in the interior and $x$ on the boundary of $M$.
The groupoid fiber $\G_x^{\pm}$ for $x$ in the interior $M_0$ trivializes to the pair groupoids and this case is thus immediate. We need to consider the case
of a point on the boundary of $M$.
Assume that $x \in F$ for some open face $F$ of $M$ which is incident to $Y$ (i.e. shares a hyperface with $Y$). 
By the local triviality property of groupoids (see \cite{MONT}) we have
\[
\G_x^{+} \cong F \times \Rr_{+}^{\ast}. 
\]

The same follows by definition for $\X$, i.e.
\[
\X_x \cong F_{ij} \times \Rr_{+}^{\ast}
\]

where $F_{ij}$ denotes the face of $F$ such that $F_{ij} = \overline{F} \cap Y$. 
Via the definition of the Lie manifold with boundary (cf. \cite{AIN}) we obtain that the component $F \times \Rr_{+}^{\ast}$
is the interior of a manifold with boundary. In particular we see that
\[
\partial_{reg} (F \times \Rr_{+}^{\ast}) = \partial_{reg}(\overline{F} \times \Rr_{+}^{\ast}) \cong F_{ij} \times \Rr_{+}^{\ast}.
\]

In summary, we obtain 
\begin{align*}
\partial_{reg} \G_x^{+} &\cong \begin{cases} 
                      Y_0 &\text{for} \ x \in X_0 \\
		      F_{ij} \times \Rr_{+}^{\ast} &\text{for} \ x \ \text{incident to some} \ F \\
		      \emptyset &\text{otherwise}
                     \end{cases} \\
&\cong \X_x. 
\end{align*}
Hence condition \emph{iv)} holds. 
\end{proof}

\begin{Rem}
On the given Lie manifold with boundary we can always restrict the groupoid $\G$ integrating the Lie structure $2\V$ to a small tubular neighborhood $Y \subset \U \subset M$
such that the faces of $Y$ are in bijective correspondence with the faces of $\U$ as specified in the introduction.
Then it is not hard to verify that there is a bibundle correspondence between $\G_Y^Y = \G_{\partial}$ and $\G$ implemented
by $\X = r^{-1}(Y), \ \Xop = s^{-1}(Y)$. Condition \emph{iii)} follows since $\G_{M_0} \cong M_0 \times M_0$ by assumption. 
Also setting $\G^{\pm} = \G_X^X$ we see that condition \emph{i)} holds. By the previous proof we also obtain that condition \emph{iv)} holds
whenever the groupoid $\G$ is locally trivial. Therefore there is in fact a boundary structure for any Lie manifold with boundary for which there is an
integrating groupoid with the local triviality property.
\label{Rem:bdystruct}
\end{Rem}

\section{Operators on groupoids}

The next goal is to define potential, trace and singular Green operators on the groupoid level.
These operators should be equivariant families of operators on the fibers, similar to the case of
pseudodifferential operators on groupoids.
The singular Green, trace and potential operators are ordinarily defined so as to act like pseudodifferential operators in the cotangent direction and as convolution operators
in the normal direction. 
This somewhat complicated behaviour is difficult to realize in the groupoid setting. 
We start from a different but equivalent definition. 
Our approach is inspired by the ordinary case of a smooth, compact manifold with boundary, cf. \cite{G}. 
Here the trace, potential and singular Green operators are extended to the double of the manifold and
can be understood as conormal distributions with rapid decay along the normal direction.
In our general setting we would therefore like to consider conormal distributions on $Y \times M$ and $M \times Y$ as well
as $M \times M$.
Since we are working in the setting of manifolds with corners we will desingularize these manifolds using Lie groupoids.
This is where the previously introduced notion of a boundary structure enters.
For the cases of $M \times M$ and $Y \times Y$ this is realized through the groupoids $\G$ and $\G_{\partial}$ respectively and the pseudodifferential
operators on groupoids.
We introduce additional blowups $\X$ and $\Xop$ with good properties (fibered over the manifolds $Y$ and $M$) with regard to $\G$ and $\G_{\partial}$. 
Then we define the trace, potential and singular Green operators as distributions on these spaces and $\G$ conormal to the diagonal $\Delta_Y$.

\subsection{Actions}
\label{actions}

\textbf{\emph{From now on we fix:}} A boundary structure $(\G, \ \G_{\partial}, \ \G^{\pm}, \ \X, \ \Xop, \ \flip)$ adapted
to our Lie manifold $(X, \ \V)$ with boundary $Y$ and its double $(M, \ 2\V)$. 
We then fix the groupoid actions which are summarized in the following picture. 
In the first column the dotted arrows indicate morphisms in the category of Lie groupoids (correspondences) which are implemented by the actions in the second column. 
\[
\begin{tikzcd}[every label/.append style={swap}]
\G_{\partial} \ar[dotted]{rr} \ar[symbol=\circlearrowleft]{rd} & & \G \ar[dotted]{rr} \ar[symbol=\circlearrowright]{ld} \ar[symbol=\circlearrowleft]{rd} & & \G_{\partial} \ar[symbol=\circlearrowright]{ld} \\[-2\jot]
& \ar{ld}{p} \X \ar{rr}{\flip} \ar{rd}{q} & & \ar{ld}{p^t} \Xop \ar{rd}{q^t} & \\
Y & & M & & Y 
\end{tikzcd}
\]

Fix also Haar systems on the groupoids and fibered spaces as follows.
\begin{align*}
& \G : \{\mu_x\}_{x \in M}, \ \X : \{\lambda_x\}_{x \in M}, \\
& \G_{\partial} : \{\mu_{y}^{\partial}\}_{y \in Y}, \ \Xop : \{\lambda_{x}^t\}_{x \in M}.
\end{align*}

In each case the system is a (left / right)-Haar system if the corresponding action is a (left / right)-action.

\subsection{Local charts}
\label{loccharts}


In order to define the operators on groupoids and actions as given in the last section we have to introduce the local charts.
The charts are given by diffeomorphisms which preserve the $s$-fibers, see also \cite{PAT2}, p. 3.


Fix the dimensions $n = \dim M = \dim M_0, \ n-1 = \dim Y = \dim Y_0$. 

\begin{itemize}
\item A chart of $\G$ is an open subset $\Omega \subset \G$ which is diffeomorphic to two open subsets of $\Gop \times \Rr^n$.
Choose two open subsets $V_s \times W_s$ and $V_r \times W_r$. 
Then choose two diffeomorphisms $\psi_s \colon \Omega \to V_s \times W_s$ and $\psi_r \colon \Omega \to V_r \times W_r$.
Additionally, we require that these diffeomorphisms are \emph{fiber-preserving} in the sense that $s(\psi_s(x,w)) = x$ for $(x, w) \in V_s \times W_s$ and $r(\psi_r(x, w)) = x$ for $(x, w) \in V_r \times W_r$.
Hence we have the following commuting diagrams:
\begin{align*}
\xymatrix{
& r(\Omega) \times W_r \ar[ld] & \ar[l]_-{\psi_r} \ar[ld] \Omega \ar[rd] \ar[r]^-{\psi_s} & s(\Omega) \times W_s \ar[rd] & \\
V_r & \ar[l] r(\Omega) & & s(\Omega) \ar[r] & V_s
} 
\end{align*}

\item Similarly, the charts for $\X$ are given by the sets of the form $\tilde{\Omega} = \Omega \cap \X$ for charts $\Omega$ of $\G$ fitting into the following commuting diagrams:
\begin{align*}
\xymatrix{
& V_q \times W_q \ar[ld] & \ar[l] \ar[ld] \tilde{\Omega} \ar[rd] \ar[r] & V_{p^t} \times W_{p^t} \ar[rd] & \\
W_{q} & \ar[l] q(\tilde{\Omega}) & & p^t(\tilde{\Omega}) \ar[r] & V_{p^t} 
} 
\end{align*}

\item Analogously, $\tilde{\Omega} \subset \Xop$ are charts with the actions reversed and hence in this case we have the commuting diagrams:
\begin{align*}
\xymatrix{
& V_p \times W_p \ar[ld] & \ar[l] \ar[ld] \tilde{\Omega} \ar[rd] \ar[r] & V_{q^t} \times W_{q^t} \ar[rd] & \\
W_{p} & \ar[l] p(\tilde{\Omega}) & & q^t(\tilde{\Omega}) \ar[r] & V_{q^t}  
} 
\end{align*}
\end{itemize}




\begin{Def}
\emph{i)} A family $T = (T_x)_{x \in M}$ of operators $T_x \colon C_c^{\infty}(\G_x) \to C_c^{\infty}(\X_x)$ is a \emph{differentiable family of trace type} iff the following holds.
Given any chart $\Omega \subset \G$ with fiber preserving diffeomorphism, $s(\Omega) \sim \Omega \times W$ for some $W \subset \Rr^n$
open. Moreover for $\tilde{\Omega} := \Omega \cap \X$ such that $\tilde{W} := W \cap \Rr^{n-1}$ we have a fiber-preserving
diffeomorphism $q(\tilde{\Omega}) \sim \tilde{\Omega} \times \tilde{W}$, and for each $\varphi \in C_c^{\infty}(\Omega), \ \tilde{\varphi} \in C_c^{\infty}(\tilde{\Omega})$ the operator
$\tilde{\varphi} T \varphi$ has a Schwartz kernel
\[
k \in I^m(s(\Omega) \times \tilde{W} \times W, \Delta_{\tilde{W}}) \cong C^{\infty}(s(\Omega)) \piotimes I^m(\tilde{W} \times W, \Delta_{\tilde{W}}).
\] 

The operator $\tilde{\varphi} T_x \varphi$ for each $x \in s(\Omega)$ corresponds to the Schwartz kernel $k_x$ 
via the diffeomorphisms $\X_x \cap \tilde{\Omega} \cong \tilde{W}$ and $\G_x \cap \Omega \cong W$. 

\emph{ii)} Analogously, we define a family $K = (K_x)_{x \in M}$ of operators $K \colon C_c^{\infty}(\X_x^t) \to C_c^{\infty}(\G_x)$ with the charts reversed.
This is called \emph{differentiable family of potential type}.

\emph{iii)} A \emph{differentiable family of singular Green type} $(G_x)_{x \in M}$ is a family of operators $G_x \colon C_c^{\infty}(\G_x) \to C_c^{\infty}(\G_x)$ defined as follows.
Given any chart $\Omega \subset \G$ with fiber preserving diffeomorphism $s(\Omega) \sim \Omega \times W$ for some $W \subset \Rr^n$ open
and $\tilde{W} = W \cap \Rr^{n-1}$. 
Then for each $\varphi \in C_c^{\infty}(\Omega)$ the operator $\varphi G \varphi$ has a Schwartz kernel
\[
k \in I^m(s(\Omega) \times W \times W, \Delta_{\tilde{W}}) \cong C^{\infty}(s(\Omega)) \piotimes I^m(W \times W, \Delta_{\tilde{W}}).
\] 

Furthermore, $\varphi G_x \varphi$ for each $x \in s(\Omega)$ corresponds to the Schwartz kernel $k_x$
via the diffeomorphism $\G_x \cap \Omega \cong W$. 

\label{Def:operators}
\end{Def}



Fix the following operations
\begin{align*}
& \mu_{\G} \colon \G \times \G \to \G, (\gamma, \eta) \mapsto \gamma \eta^{-1}, \\
& \mu \colon \X \times \G \to \X, (z, \gamma) \mapsto z \cdot \gamma^{-1}, \\
& \mu^t \colon \G \times \X^t \to \X, (\gamma, z) \mapsto \gamma^{-1} \cdot z
\end{align*}

whenever defined.

A trace type family $T$ has a family of Schwartz kernels $(k_x^T)_{x \in M}$.
Define the support of $T$ as 
\[
\supp(T) = \overline{\bigcup_{x \in M} \supp(k_x^T)}.
\]

The \emph{reduced support} of $T$ is written
\[
\supp_{\mu}(T) = \mu(\supp(T)). 
\]

The analogous definitions for potential type operators $K$ and Green type operators $G$ are given by
\[
\supp_{\mu^t}(K) = \mu^t(\supp(K)), \ \supp_{\mu_{\G}}(G) = \mu_{\G}(\supp(G)). 
\]

\begin{Def}
\begin{itemize}
\item An \emph{extended trace operator} is a differentiable family $T = (T_x)_{x \in M}$ of trace type which is a right $\X$-operator (see Definition \ref{Xop}, p. \pageref{Xop}) such that 
the reduced support of $T$ is a compact subset of $\X$. 

\item An \emph{extended potential operator} is a differentiable family $K = (K_x)_{x \in M}$ of potential type which is a left $\X^t$-operator such that
the reduced support of $K$ is a compact subset of $\X$. 

\item An \emph{extended singular Green operator} is a differentiable family $G = (G_x)_{x \in M}$ of singular Green type which 
is equivariant and such that the reduced support of $G$ is a compact subset of $\G$. 
\end{itemize}

\label{Def:extops}
\end{Def}



% ...

\begin{Rem}
\emph{i)} Since we also have a right action of $\G$ on $\X$ and $\X$ is diffeomorphic (via $\flip$) to $\Xop$ we obtain that being a left $\Xop$-operator
is equivalent to the equivariance condition with regard to the right action of $\G$ on $\X$ given in equation \eqref{inv} on p. \pageref{inv}. 
Hence a potential operator is also a right operator with regard to $\X$ in this sense, which furnishes by the proof of Prop. \ref{Prop:reduced} a reduced kernel for extended potential operators.

\emph{ii)} Note that we obtain the reduced kernels for pseudodifferential operators on $\G$ and extended singular Green operators 
with an argument completely analogous to the proof of Prop. \ref{Prop:reduced}.

\label{Rem:extops}
\end{Rem}

\begin{Prop}
\emph{i)} Given an extended trace operator $T$ the reduced kernel $k_T$ (see Proposition \eqref{Prop:reduced}, p. \pageref{Prop:reduced}) is a compactly supported distribution on $\X$ conormal to $\Delta_Y$. 

\emph{ii)} Analogously an extended potential operator $K$ has reduced kernel $k_K$ a compactly supported distribution on $\Xop$ conormal to $\Delta_Y$.
Furthermore, $K$ is the adjoint of an extended trace operator.

\emph{iii)} An extended singular Green operator $G$ has a reduced kernel $k_G$ being a compactly supported distribution on $\G$ conormal to $\Delta_Y$.

\label{Prop:conormal}
\end{Prop}

\begin{proof}
We give a proof of conormality for the case \emph{i)} of extended trace operators. The other cases are the same.

Given a family of Schwartz kernels for $(k_x^T)_{x \in M}$ contained in $I^m(\X_x \times \G_x, \X_x)$ for each $x \in M$. 
Rewrite this as
\[
k_x^T = \mu^{\ast}(k_T)_{|\X_x \times \G_x}, \ \X_x \subset \G_x \ \text{(transversal)}.
\]

Here $\mu$ is the map $\X \ast \G \ni (z, \gamma) \mapsto z \cdot \gamma^{-1} \in \X$ and 
\[
\scal{\mu^{\ast}(k_T)}{f} = \left \langle k_T(z), \ \int_{w = z \cdot \gamma} f(w, \gamma) \right \rangle.
\]

Then we need to show that: $\singsupp(k_T) \subset Y \cong \Delta_Y$.

To this end let $z \in \X \setminus \Delta_Y$ and $\varphi \in C_c^{\infty}(\X)$ a cutoff function such that 
$\varphi$ is equal to $1$ in a neighborhood of $\Delta_Y$ and equal to $0$ in a neighborhood containing $z$. 
Then
\[
\mu^{\ast}((1 - \varphi) k_T) = (1 - \varphi \circ \mu) \mu^{\ast}(k_T)
\]

restricted to $\X_x \times \G_x$ yields $(1 - \varphi \circ \mu) k_x^T$ and this is $C^{\infty}$ because $\singsupp(k_x^T) \subset \Delta_x \cong \X_x \subset \X_x \times \G_x$
by definition.
Hence $(1 - \varphi \circ \mu) \mu^{\ast}(k_T)$ is $C^{\infty}$, but this implies that $(1 - \varphi) k_T$ is smooth as well.
This proves conormality.

Finally, we show that a trace operator is the adjoint of a potential operator and vice versa.
Let $T = (T_x)_{x \in M}$ be an extended trace operator and let $(k_x^T)_{x \in M}$ be the corresponding family of Schwartz kernels.
The adjoint $T^{\ast} = (T_x^{\ast})_{x \in M}$ is given by $T_x^{\ast} \colon C_c^{\infty}(\X_x) \to C_c^{\infty}(\G_x)$ such that for $u \in C_c^{\infty}(\X_x)$ we have
\[
(T_x^{\ast} u)(\gamma) = \int_{\X_{s(\gamma)}} \overline{k_x^T(z, \gamma)} u(z) \,d\lambda_{s(\gamma)}(z).
\]

Define the family of operators $K = (K_x)_{x \in M}$ by $K = T^{\ast}$ and $k_x^K(\gamma, z) := \overline{k_x^T(z, \gamma)}$. 
We obtain a family $(k_x^K)_{x \in M}$ of distributions on $\G_x \times \X_x$ conormal to $\Delta_x \cong \X_x$ for each $x \in M$.
In addition $K$ is equivariant with regard to the right action of $\G$ which is by remark \ref{Rem:extops} equivalent to being a left $\X^t$-operator. 
Hence $K$ is an extended potential operator. 
The same argument shows that the adjoint of an extended potential operator is an extended trace operator. 
\end{proof}

\begin{Not}
We fix the notation for the reduced kernels and denote by $I_c^m(\X, \Delta_Y)$ the space of reduced kernels
of extended trace operators of order $m$, by $I_c^m(\Xop, \Delta_Y)$ the reduced kernels of extended potential operators
of order $m$ and by $I_c^m(\G, \Delta_Y)$ the space of reduced kernels of singular Green operators of order $m$. 
For the pseudodifferential operators on $\G$ of order $m$ we use the notation $\Psi^m(\G)$ for the space of operators and $I_c^m(\G, \Delta_M)$ for the reduced kernels.
With the Schwartz kernel theorem it can be proven that the spaces $\Psi^m(\G)$ and $I_c^m(\G, \Delta_M)$ are isomorphic, see \cite{NWX}, p. 24.
\label{Not:reduced}
\end{Not}

\begin{Rem}
We will also use the notation
\begin{align*}
& \Trace^{m,0}(\G, \G_{\partial}) := \J_{tr} \circ I_{c}^{m}(\X, \Delta_Y), \ \Pot^{m,0}(\G, \G_{\partial}) := \J_{pot} \circ I_{c}^{m}(\Xop, \Delta_Y), \\
& \Green^{m,0}(\G, \G_{\partial}) := \J_{gr} \circ I_{c}^{m}(\G, \Delta_Y)
\end{align*}

for these classes of extended trace, potential and singular Green operators, respectively. The $\J_{\cdot}$ in each case 
are the appropriate isomorphisms from the Schwartz kernel theorem.

Hence the operators defined previously act as follows. The mapping
\[
\J_{tr} \colon I_{c}^{m}(\X, \Delta_Y) \to \Trace^{m,0}(\G, \G_{\partial}) \subset \Hom(C_c^{\infty}(\G), C_c^{\infty}(\X))
\]

is for $z \in \X, \ k_T \in I_{c}^{m}(\X, \Delta_Y)$ given by 
\[
(\J_{tr}(k_T)u)(z) = \int_{\G_{q(z)}} k_T(z \cdot \gamma^{-1}) u(\gamma) \,d\mu_{q(z)}(\gamma).
\]

Analogously for the potential operators we have
\[
\J_{pot} \colon I_{c}^{m}(\Xop, \Delta_Y) \to \Pot^{m,0}(\G, \G_{\partial}) \subset \Hom(C_c^{\infty}(\Xop), C_c^{\infty}(\G))
\]

which for $\gamma \in \G, \ k_K \in I_{c}^{m}(\Xop, \Delta_Y)$ is given by
\[
(\J_{pot}(k_K) u)(\gamma) = \int_{\X_{r(\gamma)}^t} k_K(\gamma^{-1} \cdot z) u(z) \,d\lambda_{r(\gamma)}^t(z).
\]

Lastly, for the singular Green operators
\[
\J_{gr} \colon I_{c}^{m}(\G, \Delta_Y) \to \Green^{m,0}(\G, \G_{\partial}) \subset \Hom(C_c^{\infty}(\G), C_c^{\infty}(\G))
\]

we have for $\gamma \in \G, \ k_G \in I_{c}^{m}(\G, \Delta_Y)$
\[
(\J_{gr}(k_G) u)(\gamma) = \int_{\G_{s(\gamma)}} k_G(\gamma \eta^{-1}) u(\eta) \,d\mu_{s(\gamma)}(\eta).
\]
\label{Rem:Schwartz}
\end{Rem}


With any fibered space, longitudinally smooth via an action of a nice enough groupoid one can associate
an equivariant calculus of pseudodifferential operators.
We want to define such a calculus on $\X$ and $\Xop$. 
The following definition can in somewhat greater generality also be found in \cite{PAT2}. 
\begin{Def}
A family of pseudodifferential operators of order $m$ on $\X$ is defined as $S = (S_x)_{x \in M}$ such that

\emph{i)} each $S_x \colon C^{\infty}(\X_x) \to C^{\infty}(\X_x)$ is contained in $\Psi^m(\X_x)$. 

\emph{ii)} For each chart of $\X$ given by $\Omega \sim q(\Omega) \times W$ there is a smooth function
$a \colon q(\Omega) \to S^m(T^{\ast} W)$ such that for each $x \in q(\Omega)$ we have
\[
S_{x|\Omega \cap \X_x} = a_x(y, D_y)
\]

via identifying $\Omega \cap \X_x$ with $W$. 
Here $a_x(y, \xi) = a(x)(y, \xi)$. 
We denote by $\Psi^m(\X)$ the set of pseudodifferential families on $\X$. 
\label{Def:bdypsdo}
\end{Def}

This leads immediately to a definition of equivariant pseudodifferential operators on $\X$ and $\Xop$.
\begin{Def}
The space of equivariant pseudodifferential operators $\Psi^{\bullet}(\X)^{\G}$ on $\X$ consists of elements $S = (S_x)_{x \in M}$ of $\Psi^{\bullet}(\X)$ such that the
following equivariance condition holds
\[
R_{\gamma^{-1}} S_{r(\gamma)} R_{\gamma} = S_{s(\gamma)}, \ \gamma \in \G. 
\]

By analogy we define the equivariant pseudodifferential operators $\presuper{\G}{\Psi^{\bullet}(\Xop)}$ on $\Xop$ coming from the left action of $\G$.
The equivariance condition in this case is given as in Definition \ref{Xop}, \emph{ii)} on p. \pageref{Xop}. 
\label{Def:bdypsdo2}
\end{Def}

The operators defined here are in each case families parametrized over the double $M$. 
We have to clarify what role the pseudodifferential operators defined on $\G_{\partial}$ play.


\begin{Prop}
We have the following exact sequence
\[
\xymatrix{
C_Y^{\infty}(M) \Psi^{\bullet}(\X)^{\G} \ar@{>->}[r] & \Psi^{\bullet}(\X)^{\G} \ar@{->>}[r]^{\R_Y^{\G}} & \Psi^{\bullet}(\G_{\partial})
}
\]

where $\R_Y^{\G}$ is a well-defined restriction of families $(S_x)_{x \in M} \mapsto (S_y)_{x \in Y}$. 
Here $C_Y^{\infty}(M)$ are the smooth functions on $M$ that vanish on $Y$. 
\label{Prop:bdypsdo}
\end{Prop}

\begin{proof}
First note that $\G_{\partial}$ acts (from the left and the right) on itself. Extend this action to the set of families
$(S_y)_{y \in Y}$ with $S_y \in \Psi^{\ast}((\G_{\partial})_{y})$. 
Invariance under this action is just the usual equivariance condition for pseudodifferential operators.
Together with the uniform support condition we therefore recover the class of pseudodifferential operators, denoted $\Psi^{\ast}(\G_{\partial})$. 

The exactness of the sequence
\[
\xymatrix{ 
C_Y^{\infty}(M) \Psi^{\bullet}(\X) \ar@{>->}[r] & \Psi^{\bullet}(\X) \ar@{->>}[r]^{\R_Y} & \Psi^{\bullet}(\X_{|Y}) 
}
\]

for the restriction operator $\R_Y$ defined by $\R_Y((S_x)_{x \in M}) = (S_y)_{y \in Y}$ is immediate. 
Here note that 
\[
\X_{|Y} = q^{-1}(Y) = r^{-1}(Y) \cap s^{-1}(Y) = \G_Y^Y = \G_{\partial}
\]

by assumption.

Note also that by the previous remarks $\Psi^{\ast}(\X_{|Y})^{\G_Y^Y} \cong \Psi^{\ast}(\G_{\partial})$. 
This furnishes the exact sequence of equivariant pseudodifferential operators with a well-defined restriction map $\R_Y^{\G}$ 
\[
\xymatrix{
C_Y^{\infty}(M) \Psi^{\bullet}(\X)^{\G} \ar@{>->}[r] & \Psi^{\bullet}(\X)^{\G} \ar@{->>}[r]^-{\R_Y^{\G}} & \Psi^{\bullet}(\X_{|Y})^{\G_Y^Y} \cong \Psi^{\ast}(\G_{\partial}). 
}
\]
\end{proof}

\section{Compositions}
\label{comps}

In order to prove the main Theorem we first establish a Lemma about compositions of conormal distributions.
\begin{Lem}
The classes of extended Boutet de Monvel operators are closed under compositions induced by groupoid actions and convolution.
More precisely we have the following compositions:
\begin{align}
& \ast \colon I_{c}^{m_1}(\Xop, \Delta_Y) \times I_{c}^{m_2}(\X, \Delta_Y) \to I_{c}^{m_1 + m_2}(\G, \Delta_Y), \label{C1} \\
& \ast \colon I_{c}^{m_1}(\X, \Delta_Y) \times I_{c}^{m_2}(\Xop, \Delta_Y) \to \Psi^{m_1 + m_2}(\X)^{\G}, \label{C2} \\
& \ast \colon \Psi^{m_1}(\X)^{\G} \times I_{c}^{m_2}(\X, \Delta_Y) \to I_{c}^{m_1 + m_2}(\X, \Delta_Y), \label{C3} \\
& \ast \colon \Psi^{m_1}(\G) \times I_{c}^{m_2}(\G, \Delta_Y) \to I_c^{m_1 + m_2}(\G, \Delta_Y), \label{C4} \\
& \ast \colon I_c^{m_1}(\G, \Delta_Y) \times \Psi^{m_2}(\G) \to I_c^{m_1 + m_2}(\G, \Delta_Y), \label{C5} \\
& \ast \colon I_c^{m_1}(\G, \Delta_Y) \times I_c^{m_2}(\Xop, \Delta_Y) \to I_c^{m_1 + m_2}(\Xop, \Delta_Y), \label{C6} \\
& \ast \colon I_c^{m_1}(\X, \Delta_Y) \times \Psi^{m_2}(\G) \to I_c^{m_1 + m_2}(\X, \Delta_Y), \label{C7} \\
& \ast \colon \Psi^{m_1}(\G) \times I_c^{m_2}(\Xop, \Delta_Y) \to I_c^{m_1 + m_2}(\Xop, \Delta_Y), \label{C8} \\
& \ast \colon I_c^{m_1}(\Xop, \Delta_Y) \times \presuper{\G}{\Psi^{m_2}(\Xop)} \to I_c^{m_1 + m_2}(\Xop, \Delta_Y), \label{C9} \\
& \ast \colon I_c^{m_1}(\X, \Delta_Y) \times I_c^{m_2}(\G, \Delta_Y) \to I_c^{m_1 + m_2}(\X, \Delta_Y). \label{C10}
\end{align}

\label{Lem:conormal}
\end{Lem}

\begin{proof}

We have the equivalences \eqref{C8} $\Leftrightarrow$ \eqref{C7}, \eqref{C9} $\Leftrightarrow$ \eqref{C3} and \eqref{C10} $\Leftrightarrow$ \eqref{C6} by Prop. \ref{Prop:conormal}, \emph{ii)}. 
Since the argument in each case goes along the same lines we only treat the first 3 cases of compositions exemplarily.


\emph{i)} We consider first the case of the composition \eqref{C1}. 
Consider a family of extended trace operators $T = (T_x)_{x \in M}$ and extended potential operators $K = (K_x)_{x \in M}$. 
Denote the corresponding family of Schwartz kernels by $k_x^T \in I^{m_1}(\X_x \times \G_x, \X_x)$ as well as
$k_x^K \in I^{m_2}(\G_x \times \X_x^t, \X_x^t)$ for $x \in M$. 

We make the following computation involving an interchange of integration we still have to justify via a reduction to local coordinates.
Let $\gamma \in \G_x$ then
\begin{align*}
(K_x \cdot T_x) u(\gamma) &= \int_{\X_x^t} k_x^K(\gamma, z) (T_x u)(z) \,d\lambda_x^t(z) \\
&= \int_{\X_x^t} \int_{\G_x} k_x^K(\gamma, z) k_x^T(z, \eta) u(\eta) \,d\mu_x(\eta) \,d\lambda_x^t(z) \\
&= \int_{\G_x} k_x^{K \cdot T}(\gamma, \eta) u(\eta) \,d\mu_x(\eta).
\end{align*}

The kernels of the composition $k_x^{K \cdot T}$ would therefore take the form
\[
k_x^{K \cdot T}(\gamma, \eta) = \int_{\X_x^t} k_x^K(\gamma, z) k_x^T(z, \eta)\,d\lambda_x^t(z).
\]

This corresponds to the convolution of reduced kernels $k_K \ast k_T$ which is immediately defined from
the actions.

First we check the support condition of the composed operator. 
The reduced support is compact via the inclusion
\[
\supp_{\mu}(K \cdot T) \subset \mu_{\G}(\supp_{\lambda}(K) \times \supp_{\lambda^t}(T)).
\]

Here the inversion of elements in the spaces $\X$ and $\X^t$ is performed inside the groupoid $\G$ where it is always defined.

Fix the projections
\[
p_1 \colon \X \times \G \to \X, \ p_2 \colon \X \times \G \to \G.
\]

Then by the uniform support condition the family $T = (T_x : x \in M)$ is in particular properly supported.
This means for compact sets $K_1 \subset \G, \ K_2 \subset \X$ we have that
\[
p_i^{-1}(K_i) \cap \supp(k_T) \subset \X \times \G, i = 1,2
\]

is compact. We make use of this property for the following argument.

Next we check the smoothness property of compositions. 
Let $f \in C_c^{\infty}(\G)$ be given, we will show that $Tf \in C_c^{\infty}(\X)$.
Assume that $T = (T_x : x \in M)$ has a Schwartz kernel $k_T$ contained in
$I^{-\infty}(\X, \Delta_Y) = \bigcap_m I^m(\X, \Delta_Y)$. 
Then $k_T$ is $C^{\infty}$ on the closed subset $\{(z, \gamma) : q(z) = s(\gamma)\}$ of $\X \times \G$ and
via fiber preserving diffeomorphisms we obtain a $C^{\infty}$-atlas.
The function $T f$ yields a smooth function because we integrate the kernels $k_x^T$ which are smooth functions.
Hence we can interchange integration and differentiation.
Therefore $Tf \in C_c^{\infty}(\X)$ for $k_T \in I^{-\infty}(\X, \Delta_Y)$. 

Consider a general extended trace operator $T$.
Let $(\gamma, z) \in \G_x \times \X_x$ and $\Omega \subset \G$ be a chart with fiber preserving diffeomorphism
$\Omega \sim s(\Omega) \times W$. We can assume that $W \subset \Rr^n$ is convex, open, $0 \in W$ and that $\Omega$ is a neighborhood 
of $\gamma$ such that $(x, 0)$ gets mapped to $\gamma$ via the diffeomorphism. 
We also set $\tilde{\Omega} = \Omega \cap \X \subset \X$ and $\tilde{W} = W \cap \Rr^{n-1}$ with a fiber preserving
diffeomorphism $\tilde{\Omega} \sim q(\tilde{\Omega}) \times \tilde{W}$ (recall the fact that $q = s_{|\G_M^Y}$ by assumption).
By the previous remarks the family $T$ is properly supported, which implies in particular that each $T_x$ is properly supported for $x \in M$.
Hence we obtain that the kernels $k_x^T$ of $T$ satisfy the support estimate
\begin{align*}
& p_1^{-1}\left(q(\tilde{\Omega}) \times \frac{\tilde{W}}{2}\right) \cap p_2^{-1}\left(s(\Omega) \times \frac{W}{2}\right)
\cap \overline{\bigcup_x \supp(k_x^T)} \subset \left(q(\tilde{\Omega}) \times \frac{3\tilde{W}}{2} \right) \times \left(s(\Omega) \times \frac{3W}{2}\right)
\end{align*}

therefore the fact that $Tf \in C_c^{\infty}(\X)$ reduces to a computation in local coordinates.

Similar reasoning applies to potential operators. Using the same argument as above we deduce that $K f \in C_c^{\infty}(\G)$ for $f \in C_c^{\infty}(\X)$ if
$K$ is smoothing.
For a general $K$ we note that each $K_x$ is properly supported for $x \in M$ and hence we obtain that the kernels $k_x^K$ of $K$ satisfy the support estimate
\begin{align*}
& p_1^{-1}\left(q(\tilde{\Omega}) \times \frac{\tilde{W}}{2} \right) \cap p_2^{-1}\left(s(\Omega) \times \frac{W}{2} \right) \cap \overline{\bigcup_{x} \supp(k_x^K)} \subset \left(q(\tilde{\Omega}) \times \frac{3 \tilde{W}}{2} \right) \times \left(s(\Omega) \times \frac{3W}{2}\right).
\end{align*}

The smoothness of $K f$ reduces to a computation in local coordinates.
Consider the general composition $K \cdot T$ for $T$ an arbitrary extended trace operator and $K$ an arbitrary extended potential operator. 
Then make the suitable support estimates as above to show that $K_x \cdot T_x$ are compositions of smooth families of conormal distributions which act on the sets
$W \subset \Rr^n, \ \tilde{W} \subset \Rr^{n-1}$.
It then follows from a general theorem of H\"ormander about compositions of conormal distributions, see \cite{HIV}, Thm. 25.2.3, p. 21 that the composition $T \cdot K$
yields a family of operators each of which is a conormal distribution in the sense of Definition \ref{Def:operators}, \emph{iii)}. 

We therefore obtain a properly supported family $K \cdot T$ which is by the above argument uniformly supported since $K$ and $T$ are each uniformly supported.
Finally, we check the equivariance property for the new family of operators $G = K \cdot T$. 
Since $(R_{\gamma})^{-1} = R_{\gamma^{-1}}$ we can write the equivariance condition from Definition \ref{Xop} in the form
\[
T_{r(\gamma)} R_{\gamma} = R_{\gamma} T_{s(\gamma)}, \ \forall \ \gamma \in \G. 
\]

Since $K$ is equivariant with regard to the right action of $\G$ on $\X$ by Remark \ref{Rem:extops}, \emph{i)}, the equivariance condition for $K$ reads
\[
R_{\gamma^{-1}} K_{r(\gamma)} = K_{s(\gamma)} R_{\gamma^{-1}}, \ \forall \ \gamma \in \G.
\]

For $\gamma \in \G$ we calculate
\begin{align*}
& R_{\gamma^{-1}} (K \cdot T)_{r(\gamma)} R_{\gamma} = R_{\gamma^{-1}} (K_{r(\gamma)} \cdot T_{r(\gamma)}) R_{\gamma} = R_{\gamma^{-1}} K_{r(\gamma)} R_{\gamma} T_{s(\gamma)} \\ 
&= K_{s(\gamma)} R_{\gamma^{-1}} \cdot R_{\gamma} T_{s(\gamma)} = K_{s(\gamma)} ((R_{\gamma})^{-1} \cdot R_{\gamma}) T_{s(\gamma)} = K_{s(\gamma)} \cdot T_{s(\gamma)} \\
&= (K \cdot T)_{s(\gamma)} 
\end{align*}

and hence $K \cdot T$ has the required equivariance property with regard to the right action of the groupoid $\G$. 
We have thus verified all the properties of an extended Green operator. 

\emph{ii)} Consider the next composition $T \cdot K \colon C_c^{\infty}(\X) \to C_c^{\infty}(\X)$ which is again
for $z \in \X$ and $u \in C_c^{\infty}(\X)$ given by
\begin{align*}
(T_x \cdot K_x) u(z) &= \int_{\G_{q(z)}} k_x^T(z, \gamma) (K_x u)(\gamma)\,d\mu_x(\gamma) \\
&= \int_{\G_x} \int_{\X_x^t} k_x^T(z, \gamma) k_x^K(\gamma, w) u(w) \,d\lambda_x^t(w) \,d\mu_x(\gamma) \\
&= \int_{\X_x^t} k_x^{T \cdot K}(z, w) u(w) \,d\lambda_x^t(w). 
\end{align*}

The kernel $k_x^{T \cdot K}$ is written
\[
k_x^{T \cdot K}(z, w) = \int_{\G_x} k_x^T(z, \gamma) k_x^K(\gamma, w) \,d\mu_x(\gamma). 
\]


Now we can argue again analogously to \emph{i)} that the composition has the right support condition and via a reduction to local charts the formal computation can be made precise. 
We therefore obtain a family of kernels $k_x^{T \cdot K} \in I^{m_1 + m_2}(\X_x^t \times \X_x, \Delta_{\X_x})$. 

\emph{iii)} The third composition $S \cdot T \colon C_c^{\infty}(\G) \to C_c^{\infty}(\X)$ gives a family of 
extended trace operators. 
We obtain for $z \in \X, u \in C_c^{\infty}(\G)$
\begin{align*}
(S_x \cdot T_x)u(z) &= \int_{\X_x} k_x^S(z, w) (T_x u)(w) \,d\lambda_x(w) \\
&= \int_{\X_x} \int_{\G_x} k_x^S(z, w) k_x^T(w, \gamma) u(\gamma) \,d\mu_x(\gamma) \,d\lambda_x(w) \\
&= \int_{\G_x} k_x^{S \cdot T}(z, \gamma) u(\gamma)\,d\mu_x(\gamma).
\end{align*}

We obtain the kernel
\[
k_x^{S \cdot T}(z, \gamma) = \int_{\X_x} k_x^S(z, w) k_x^T(w, \gamma)\,d\lambda_x(w). 
\]

We proceed by making the analogous argument as in \emph{i)}, \emph{ii)}. The right support condition holds on $\Psi(\X)^{\G}$ e.g. via the identification
from \ref{Prop:bdypsdo}. The rest of the reasoning then yields a family of kernels
$k_x^{S \cdot T} \in I^{m_1 + m_2}(\X_x \times \G_x, \Delta_{\X_x})$ with the correct support condition. 
\end{proof}




\section{The calculus}

\subsection{Quantization}

Let us denote by $\B_{prop}^{m,0}(M_0, Y_0)$ the properly supported extended Boutet de Monvel operators of order $m$, defined on the interior.
In this section we introduce an algebra of extended operators $\B_{2\V}^{0,0}(M, Y)$ on the double Lie manifold $M$. 
This is defined by extending the distributional kernels in $\B_{prop}^{0,0}(M_0, Y_0)$ to take the Lie structure into account.
We also fix the actions with corresponding notation from \ref{actions}. Introduce the singular normal bundles for the inclusions $\Delta_Y \hookrightarrow \X, \ \Delta_Y \hookrightarrow \Xop$ as well as 
$\Delta_Y \hookrightarrow \G$:
\[
\N^{\X} \Delta_Y \to Y, \ \N^{\Xop} \Delta_Y \to Y, \ \N^{\G} \Delta_Y \to Y.
\]
Restricted to the interior we have by axiom \emph{iii)} in Def. \ref{Def:bdystruct} the isomorphisms
\begin{align}
& \N^{\X} \Delta_Y|_{Y_0} \cong N^{Y_0 \times M_0} \Delta_{Y_0}, \ \N^{\Xop} \Delta_Y|_{Y_0} \cong N^{M_0 \times Y_0} \Delta_{Y_0}, \ \N^{\G} \Delta_Y|_{Y_0} \cong N^{M_0 \times M_0} \Delta_{Y_0}. \label{normint}
\end{align}
Here we denote by $N^{Y_0 \times M_0} \Delta_{Y_0}$ the normal bundle to the inclusion $\Delta_{Y_0} \hookrightarrow Y_0 \times M_0$ and the same for
the others. It is not hard to see that $\N^{\X} \Delta_Y$ can be identified with $\A_{|Y}$ which is isomorphic to $A_{\partial} \oplus \N$. 

\begin{Rem}
\emph{i)} On the singular normal bundles we define the H\"ormander symbols spaces $S^m(\N^{\X} \Delta_Y^{\ast}) \subset C^{\infty}(\N^{\X} \Delta_Y^{\ast})$
such that for $U \subset Y$ open with
\[
\N^{\X} \Delta_Y^{\ast}|_{U} \cong U \times \Rr^{n-1} \times \Rr, \ K \subset U \ \text{compact}.
\]
We have the estimates for $t \in S^m(\N^{\X} \Delta_Y^{\ast})$
\[
|D_{x'}^{\alpha} D_{\xi', \xi_n}^{\beta} t(x', \xi)| \leq C_{K, \alpha, \beta} \ideal{\xi}^{m- |\beta|}, \ (x, \xi) \in K \times \Rr^{n-1} \times \Rr
\]
for each $\alpha \in \Nn_0^{n-1}, \ \beta \in \Nn_0^{n}$. Note that we have by H\"ormander's results a correspondence between the spaces of symbols on the normal bundle to a smooth manifold
and conormal distributions on the space (at least in the smooth case, cf. \cite{HIII}, Thm 18.2.11):
\[
I^{L}(\X, \Delta_Y) / I^{-\infty}(\X, \Delta_Y) \cong S^{m}(\N^{\X} \Delta_Y^{\ast}) / S^{-\infty}(\N^{\X} \Delta_Y^{\ast})
\]

where $L$ is the obligatory correction of order
\[
m = L - \frac{1}{4} \dim \mathcal{X} + \frac{1}{2} \dim \Delta_Y.
\]

We will ignore this order convention in the following discussions.
Note that our earlier definition of smooth families of operators defined as conormal distributions suggests immediately a quantization
which we state next. We can require additionally the (local) rapid decay property stated earlier, then we use the notation
$S_{\N}^{m}(\N^{\X} \Delta_Y^{\ast})$ for these symbol spaces.
Analogously the spaces 
\[
S_{\N}^m(\N^{\Xop} \Delta_Y^{\ast}) \subset C^{\infty}(\N^{\Xop} \Delta_Y^{\ast}), \ S_{\N}^m(\N^{\G} \Delta_Y^{\ast}) \subset C^{\infty}(\N^{\G} \Delta_Y^{\ast}).
\]

\emph{ii)} A second definition we will need is that of conormal distributions on the normal bundles themselves.
First given the normal and conormal bundles
\[
\pi \colon \N^{\X} \Delta_Y \to Y, \ \overline{\pi} \colon \N^{\Xop} \Delta_Y^{\ast} \to Y,
\]
(and analogously for $\N^{\Xop}, \ \N^{\G}$) define the \emph{fiberwise Fourier transform} $\Ff \colon S(\N^{\X} \Delta_Y) \to S(\N^{\X} \Delta_Y^{\ast})$
\[
\Ff(\varphi)(\xi) := \int_{\overline{\pi}(\zeta) = \pi(\xi)} e^{-i \scal{\xi}{\zeta}} \varphi(\zeta) \,d\zeta. 
\]
The inverse is given by duality
\[
\Ff^{-1}(\varphi)(\zeta) = \int_{\overline{\pi}(\zeta) = \pi(\xi)} e^{i \scal{\xi}{\zeta}} \varphi(\xi) \,d\xi, \ \varphi \in S(\N^{\X} \Delta_Y^{\ast}).
\]
Here we use the notation $S(\N^{\X} \Delta_Y), \ S(\N^{\X} \Delta_Y^{\ast})$ for the spaces of rapidly decreasing functions on the normal and conormal bundle respectively, see also \cite{S}, chapter 1.5. 
Then the spaces of conormal distributions are defined as:
\[
I^{m}(\N^{\X} \Delta_Y, \Delta_Y) := \Ff^{-1} S^m(\N^{\X} \Delta_Y^{\ast})
\]
and $I^{m}(\N^{\Xop} \Delta_Y, Y), \ I^{m}(\N^{\G} \Delta_Y, Y)$ analogously.
\label{Rem:conormal}
\end{Rem}

For the definition of the quantization rule we need some further notation. 
Let $0 < r \leq r_0$ where $r_0$ is the (positive) injectivity radius of $M$.
\begin{itemize}
\item First for the case of trace operators. We set 
\[
(\N^{\X} \Delta_Y)_r = \{v \in \N^{\X} \Delta_Y : \|v\| < r\}
\]
as well as
\[
I_{(r)}^m(\N^{\X} \Delta_Y, \Delta_Y) = I^m((\N^{\X} \Delta_Y)_r, \Delta_Y).
\]
Fix the restriction 
\[
\R \colon I_{(r)}^m(\N^{\X} \Delta_Y, \Delta_Y) \to I_{(r)}^m(N^{Y_0 \times M_0} \Delta_{Y_0}, \Delta_{Y_0}).
\]
We denote by $\Psi$ the normal fibration of the inclusion $\Delta_{Y_0} \hookrightarrow Y_0 \times M_0$
such that $\Psi$ is the local diffeomorphism mapping an open neighborhood of the zero section $O_{Y_0} \subset V \subset N^{Y_0 \times M_0} \Delta_{Y_0}$
onto an open neighborhood $\Delta_{Y_0} \subset U \subset Y_0 \times M_0$ (cf. \cite{S}, Thm. 4.1.1). 
Then we have the induced map on conormal distributions
\[
\Psi_{\ast} \colon I_{(r)}^m(N^{Y_0 \times M_0} \Delta_{Y_0}, \Delta_{Y_0}) \to I^m(Y_0 \times M_0, \Delta_{Y_0}).
\]
Also let $\chi \in C_c^{\infty}(\N^{\X} \Delta_Y)$ be a cutoff function which acts by multiplication 
\[
I^m(\N^{\X} \Delta_Y, \Delta_Y) \to I_{(r)}^{m}(\N^{\X} \Delta_Y, \Delta_Y).
\]
\item For potential operators we use the analogous notation:
\[
\R^t, \ \Psi^t, \ \F_{f}^t, \ \chi^t.
\]
\item Finally, in the singular Green case we have the induced normal fibration
\[
\Phi_{\ast} \colon I_{(r)}^m(N^{M_0 \times M_0} \Delta_{Y_0}, \Delta_{Y_0}) \to I^m(M_0 \times M_0, \Delta_{Y_0}).
\]
and the fiberwise Fourier transform
\[
\F_f^{\G} \colon I^m(\N^{\G} \Delta_Y, \Delta_Y) \to S^m(\N^{\G} \Delta_Y^{\ast}). 
\]
The restriction and cutoff is denoted by
\[
\R^{\G} \colon I_{(r)}^m(\N^{\G} \Delta_Y, \Delta_Y) \to I_{(r)}^m(N^{M_0 \times M_0} \Delta_{Y_0}, \Delta_{Y_0})
\]
and
\[
\varphi \colon I^m(\N^{\G} \Delta_Y, \Delta_Y) \to I_{(r)}^m(\N^{\G} \Delta_Y, \Delta_Y). 
\]
\end{itemize}

\begin{Def}[Quantization]
\emph{i)} Define
\[
q_{T, \chi} \colon S^m(\N^{\X} \Delta_Y^{\ast}) \to \Trace^{m,0}(M, Y) 
\]
such that for $t \in S^m(\N^{\X} \Delta_Y^{\ast})$ we have
\[
q_{T, \chi}(t) = \J_{tr} \circ q_{\Psi, \chi}(t)
\]
where
\[
q_{\Psi, \chi}(t) = \Psi_{\ast}(\R(\chi \Ff^{-1}(t))).
\]
\emph{ii)} Define 
\[
q_{K, \chi^t} \colon S^m(\N^{\Xop} \Delta_Y^{\ast}) \to \Pot^m(M, Y) 
\]
such that for $k \in S^m(\N^{\Xop} \Delta_Y^{\ast})$ we have 
\[
q_{k, \chi^t}(k) = \J_{pot} \circ q_{\Psi^t, \chi^t}(k).
\]
\emph{iii)} Define
\[
q_{G, \varphi} \colon S^m(\N^{\G} \Delta_Y^{\ast}) \to \G^{m,0}(M, Y)
\]
such that for $g \in S^m(\N^{\G} \Delta_Y^{\ast})$ we have
\[
q_{G, \varphi}(g) = \J_{pot} \circ q_{\Phi, \varphi}(g).
\]
\label{Def:quant}
\end{Def}

\begin{Prop}
The fibrations $q_{\Psi, \chi}, \ q_{\Psi^t, \chi^t}$ and $q_{\Phi,\varphi}$ define properly supported Schwartz kernels.
\label{Prop:proper}
\end{Prop}

\begin{proof}
Consider exemplarily the trace operators. 
Since $\chi \R(t)$ is properly supported we find that $q_{T,\chi}(t)$ defines a properly supported operator.
It is clear from the definition that $q_{T,\chi}(t) \colon C_c^{\infty}(M_0) \to C_c^{\infty}(Y_0)$ has the Schwartz kernel
$q_{\Psi, \chi}(t)$. 
\end{proof}

The following is easy to check. 
\begin{Prop}
The quantizations $q_{T, \chi}, \ q_{K, \chi}, \ q_{G, \chi}$ are in each case independent of the choice of cutoff functions up to smoothing errors.
\label{Prop:cutoff}
\end{Prop}

From the compactness of $M$ and $Y$ we can associate to each vector field in $2 \V$ respectively $\W$ a \emph{global flow}
\begin{align*}
& 2\V \ni V \mapsto \Phi_V \colon \Rr \times M \to M, \\
& \W \ni W \mapsto \Psi_W \colon \Rr \times Y \to Y.
\end{align*}

Then consider the diffeomorphisms evaluated at time $t = 1$
\[
\Phi(1, -) \colon M \to M \ \text{and} \ \Psi(1,-) \colon Y \to Y
\]

and fix the corresponding group actions on functions which we denote by
\begin{align*}
& 2\V \ni V \mapsto \varphi_V \colon C^{\infty}(M) \to C^{\infty}(M), \\
& \W \ni W \mapsto \psi_W \colon C^{\infty}(Y) \to C^{\infty}(Y). 
\end{align*}

The upshot of this is a definition of the suitable smoothing terms for our calculus which we state next. 
\begin{Def}
\emph{i)} The class of $\V$-trace operators is defined as
\[
\Trace_{2\V}^{m,0}(M, Y) := \Trace^{m,0}(M, Y) + \Trace_{2\V}^{-\infty, 0}(M, Y).
\]

Here $\Trace^{m,0}(M, Y)$ consists of the extended operators from the previous definition.
The residual class is defined as follows
\[
\Trace_{2\V}^{-\infty, 0}(M, Y) := \mathrm{span}\{q_{\chi, T}(t) \varphi_{V_1} \cdots \varphi_{V_k} : V_j \in 2\V, \ \chi \in C_c^{\infty}(\A_{\partial}), \ t \in S_{\N}^{-\infty}(\N^{\X} \Delta_Y^{\ast})\}.
\]

\emph{ii)} The class of $\V$-potential operators is defined in the same fashion
\[
\Pot_{2\V}^{m}(M, Y) := \Pot^{m}(M, Y) + \Pot_{2\V}^{-\infty}(M, Y)
\]

with residual class
\[
\Pot_{2\V}^{-\infty}(M, Y) := \mathrm{span}\{q_{\chi, K}(k) \psi_{W_1} \cdots \psi_{W_k} : W_j \in \W, \ \chi \in C_c^{\infty}(\A_{\partial}), \ k \in S_{\N}^{-\infty}(\N^{\Xop} \Delta_Y^{\ast})\}.
\]

\emph{iii)} Lastly, the class of $\V$-singular Green operators is defined as
\[
\Green_{2\V}^{m,0}(M, Y) := \Green^{m,0}(M, Y) + \Green_{2\V}^{-\infty, 0}(M, Y)
\]

with residual class
\[
\Green_{2\V}^{-\infty,0}(M, Y) := \mathrm{span}\{q_{\chi, G}(g) \psi_{V_1} \cdots \psi_{V_k} : V_j \in 2\V, \ \chi \in C_c^{\infty}(\A_{\partial}), \ g \in S_{\N}^{-\infty}(N^{\G} \Delta_Y^{\ast}\}.
\]

\emph{iv)} Then the calculus $\B_{2\V}^{m,0}(M, Y)$ of extended operators consists of matrices of the form
\[
A = \begin{pmatrix} P + G & K \\ T & S \end{pmatrix}
\]

for $P \in \Psi_{2\V}^{m}(M), S \in \Psi_{\W}^{m}(Y)$ and $G$ an extended singular Green operator, $K$ extended potential and $T$ extended trace operator.



\label{Def:extended}
\end{Def}

\subsection{Composition}

The restriction $\chi^+$ to the interior $\mathring{X}_0 := X_0 \setminus Y_0$ and the extension by zero operator $\chi^0$ are given on the manifold level by
\[
\xymatrix{
L^2(M_0) \ar@/^1pc/[r]^{\chi^{+}} & \ar@/-0pc/[l]^-{\chi^0} L^2(\mathring{X}_0) 
} 
\]

with $\chi^{+} \chi^{0} = \id_{L^2(\mathring{X}_0)}$ and $\chi^{0} \chi^{+}$ being a projection onto a subspace of $L^2(M_0)$.

On the groupoid level we use the same symbols since it will be clear from context which is meant.
So we define the operators
\[
\xymatrix{
L^2(\G) \ar@/^1pc/[r]^{\chi^{+}} & \ar@/-0pc/[l]^-{\chi^0} L^2(\G^{+}) 
} 
\]
with $\chi^{+} \chi^{0} = \id_{L^2(\G^{+})}$ and $\chi^{0} \chi^{+}$ being a projection onto a subspace of $L^2(\G)$.

\begin{Def}
The operator $P \in \Psi^{m}(\G)$ has the transmission property if the symbol $a \in S_{tr}^{m}(\A^{\ast})$.
Here the class of H\"ormander symbols $a \in S_{tr}^{m}(\A^{\ast})$ consists of families $a = (a_x)_{x \in M}$
such that each symbol $a_x$ has the transmission property with regard to $\X_x \subset \G_x$. 
In particular the operators $(\chi^{+} P \chi^{0})_x$ map functions smooth up to the boundary $\X_x$ to functions
which have the same property. 
\label{Def:tprop}
\end{Def}

\begin{Exa}
\begin{itemize}
 \item Notice first that if $x \in M_0$ is an interior point we have that $\G_x \cong M_0$ and we recover the transmission property
on the interior manifold $X_0$ with boundary $Y_0$.

\item In our trivial case $\G = M \times M$ and $M = 2X$, $X$ a compact manifold with boundary $\partial X = Y$ we recover the transmission property
      meaning $\Psi_{tr}^{m}(M) \cong \Psi_{tr}^{m}(\G)$.

\end{itemize}
\end{Exa}

\begin{Not}
The operation of \emph{truncation} itself is given as a linear operator.

\emph{i)} On the groupoid calculus this operator is given by
\[
\End\begin{pmatrix} C_c^{\infty}(\G) \\ C_c^{\infty}(\X) \end{pmatrix} \supset \B^{m,0}(\G, \X) \ni A = \begin{pmatrix} P + G & K \\ T & S \end{pmatrix} \mapsto \tilde{\C}(A) = \begin{pmatrix} \chi^{+} (P + G) \chi^{0} & \chi^{+} K \\
T \chi^{0} & S \end{pmatrix} \in \End\begin{pmatrix} C_c^{\infty}(\G^{+}) \\ C_c^{\infty}(\X) \end{pmatrix}. 
\]

\emph{ii)} On the extended calculus we define
\[
\End\begin{pmatrix} C_c^{\infty}(M_0) \\ C_c^{\infty}(Y_0) \end{pmatrix} \supset \B_{2\V}^{m,0}(M, Y) \ni A = \begin{pmatrix} P + G & K \\ T & S \end{pmatrix} \mapsto \C(A) = \begin{pmatrix} \chi^{+} (P + G) \chi^{0} & \chi^{+} K \\ T \chi^{0} & S \end{pmatrix} \in \End\begin{pmatrix} C_c^{\infty}(X_0) \\ C_c^{\infty}(Y_0) \end{pmatrix}.
\]
\end{Not}

\begin{Def}
\emph{i)} The Boutet de Monvel calculus on the boundary structure is defined as the set of operators for $m \leq 0$
\[
\B^{m,0}(\G^{+}, \X) := \tilde{\C} \circ \B^{m,0}(\G, \X).
\]

\emph{ii)} The class of Boutet de Monvel operators on the Lie manifold with boundary is for $m \leq 0$ defined as
\[
\B_{\V}^{m,0}(X, Y) := \C \circ \B_{2\V}^{m,0}(M, Y).
\]

The \emph{vector representation} $\tilde{\varrho}_{BM}$ is defined by
\[
A\begin{pmatrix} \varphi \circ r \\ \psi \circ r_{\partial} \end{pmatrix} = \left(\tilde{\varrho}_{BM}(A) \begin{pmatrix} \varphi \\ \psi \end{pmatrix}\right) \circ \begin{pmatrix} r \\ r_{\partial} \end{pmatrix}  
\]
for $A \in \B^{m,0}(\G^{+}, \X)$ and $\varphi \in C_c^{\infty}(X_0), \ \psi \in C_c^{\infty}(Y_0)$.
\label{Def:restricted}
\end{Def}

\begin{Proof}[of Theorem \ref{Thm:BM}]
Let $A, B \in \B^{0,0}(\G^{+}, \X)$, then $A = (A_x)_{x \in X}, \ B = (B_x)_{x \in X}$ are equivariant families of properly supported Boutet de Monvel operators 
the fibers $(\G_x^{+}, \X_x)$ which are smooth manifolds with boundary.
The composition $A \cdot B = (A_x \cdot B_x)_{x \in X}$ is on the fibers of the groupoids and the spaces $\X$ and $\Xop$. 
Using the support estimates in the proof of Lemma \ref{Lem:conormal} the calculation can be performed in local charts by reduction
to composition of properly supported operators of Boutet de Monvel type on smooth manifolds. \qed
\end{Proof}

\subsection{Representation theorem}

\begin{Thm}
Given a $\V$-boundary structure the previously defined vector representation $\varrho_{BM}$ furnishes the isomorphism
\[
\varrho_{BM} \circ \B^{m,0}(\G, \X) \cong \B_{2\V}^{m,0}(M, Y).
\]
\label{Thm:representation}
\end{Thm}

\begin{proof}
The proof is similar to the proof of Theorem 3.2. in \cite{ALN}. For further details we also refer to \cite[Theorem 7.9]{B}.
\end{proof}

\begin{Proof}[of Thm. \ref{Thm:BM2}]
From the right action of the groupoid $\G$ on $\X$ we obtain the induced action of the subgroupoid $\G^{+} \subset \G$.
The right action $\X_{|X} \ \rotatebox[origin=c]{90}{$\circlearrowright$}\ \G^{+}$ defines the equivariance of the families in the calculus $\B^{0,0}(\G^{+}, \X)$. 
The range and source maps in $\G^{+}$ are the restrictions of the range and source maps in $\G$, hence the representation $\tilde{\varrho}_{BM}$ is well-defined.
Thereby we obtain that the truncation maps are linear, equivariant maps and the following diagram commutes as linear operators
\[
\xymatrix{
\B_{\V}^{m,0}(X, Y) & \ar[l]^{\tilde{\varrho}_{BM}} \B^{m,0}(\G^{+}, \X) \\
\B_{2\V}^{m,0}(M, Y) \ar@{->>}[u]^{\C} & \ar[l]^{\varrho_{BM}} \B^{m,0}(\G, \X) \ar@{->>}[u]^{\tilde{\C}}.
}
\]
Since $\C, \tilde{\C}$ and $\varrho_{BM}$ are surjective (by Theorem \ref{Thm:representation}) we obtain the surjectivity of $\tilde{\varrho}_{BM}$ as follows.
Let $B \in \B_{\V}^{m,0}(X, Y)$ then by surjectivity of $\varrho_{BM}$ and $\C$ we lift this to an element $\tilde{B} \in \B^{m,0}(\G, \X)$.
Then $A := \tilde{\C}(\tilde{B})$ is the required preimage. By commutativity we have 
\[
\tilde{\varrho}_{BM}(A) = (\tilde{\varrho}_{BM} \circ \tilde{\C})(\tilde{B}) = (\C \circ \varrho_{BM})(\tilde{B}) = B.
\]
Hence in this case $\tilde{\varrho}_{BM}$ is surjective.
It is also immediate that it is a well-defined homomorphism of algebras. This yields the closedness under composition. \qed
\end{Proof}

\textbf{Vector bundles}

Up until now we have only considered scalar operators of Boutet de Monvel type. 
It does only require minor modifications to consider operators acting on smooth sections of smooth vector bundles, see also \cite{NWX}.
To this effect let $E_1, E_2 \to X$ be smooth vector bundles on $X$ and $J_{\pm} \to Y$ smooth vector bundles on $Y$. 
We can pull back these bundles to $\G^{+}$ via $\tilde{E}_i := r^{\ast} E_i \to \G^{+}$. 
Similarly, the actions allow us to pull back the bundles $J_{\pm}$ to $\X$ and obtain $\tilde{J}_{\pm} \to \X$. 

It is not difficult to modify our construction for operators acting on the smooth sections such that 
$A \in B_{\V}^{0,0}(X, Y; E_1, E_2, J_{\pm})$ is a continuous linear operator
\[
A \colon \begin{matrix} C^{\infty}(X, E_1) \\ \oplus \\ C^{\infty}(Y, J_{+}) \end{matrix} \to \begin{matrix} C^{\infty}(X, E_2) \\ \oplus \\ C^{\infty}(Y, J_{-}) \end{matrix}. 
\]

Similarly, $A \in \B^{0,0}(\G^{+}, \X; \tilde{E}_i, \tilde{J}_{\pm})$ is a continuous linear operator
\[
A \colon \begin{matrix} C_c^{\infty}(\G^{+}, \tilde{E}_1) \\ \oplus \\ C_c^{\infty}(\X, \tilde{J}_{+}) \end{matrix} \to \begin{matrix} C_c^{\infty}(\G^{+}, \tilde{E}_2) \\ \oplus \\ C_c^{\infty}(\X, \tilde{J}_{-}) \end{matrix}. 
\]

\section{Properties of the calculus}

\label{properties}

The following sections mark the start of the analysis of the represented side of our calculus (the algebra $\B_{\V}^{0,0}(X, Y)$).
In order to limit the size of the paper and preserve its readability most proofs are referred to the comprehensive literature
on Boutet de Monvel's calculus, where these constructions have already been given and which need only be adapted to our setup in a straightforward manner.

Set $r(\xi) := (1 + \|\xi\|^2)^{\frac{1}{2}}$, then close to the boundary (in a fixed tubular neighborhood) we define the symbol
\[
r_{-}^m(x, \xi', \xi_n) = \left(\varphi\left(\frac{\xi_n}{C \ideal{\xi'}}\right) \ideal{\xi'} - i \xi_n\right)^m
\]

for $m \in \Rr$, some constant $C > 0$ and $\varphi \in S(\Rr)$ such that $\varphi(0) = 1, \ \supp \F^{-1} \varphi \subset \Rr_{-}$.

Then $r_{-}(\xi)$ is an elliptic symbol and has the transmission property (cf. \cite{HS}, Proposition 1.3).

We construct a global classical symbol $a \in S_{cl}^1(\A^{\ast})$ as follows.
Fix a normal cover $\bigcup_{i=1}^{\infty} U_i = M$ of $M$ with local trivializations $\Psi_i \colon \A_{U_i} \cong U_i \times \Rr^n$. 

\begin{itemize}
 \item If $U_i \cap Y = \emptyset$ (an \emph{interior chart}) set $\Psi_i a(\xi) = r(\xi)$. 
 
 \item If $U_i \cap Y \not= \emptyset$ (a \emph{boundary chart}) set $\Psi_i a(\xi) = r_{-}(\xi)$. 
 
\end{itemize}

\begin{Def}
Let $a \in S_{cl}^1(\A^{\ast})$ be the order reducing symbol as defined above and set $R_{-} := [q(a^{\frac{1}{k}})^{-1}]^k$ for $k$ chosen sufficiently large
as in the proof of \cite[Corollary 4.3]{ALNV}.
\label{Def:reduction}
\end{Def}

We can use the same argument as in \cite[Proposition 1.7]{HS} to prove the following result.
\begin{Prop}
The operator $R_{+} \colon \chi^{+} R_{-}^m \chi^{0}$ extends to a linear isomorphism
\[
R_{+} \colon H_{\V}^m(X) \iso H_{\V}^{s-m}(X)
\]

for each $s \in \Rr$. 
\label{Prop:reduction}
\end{Prop}

Denote by $\gamma_Y \colon C_c^{\infty}(M_0) \to C_c^{\infty}(Y_0)$ the restriction operator $f \mapsto f_{|Y_0}$.
Then $\gamma_Y \colon H_{2 \V}^{m}(M) \to H_{\W}^{m - \frac{1}{2}}(Y)$ is continuous for each $m \in \Rr$, see \cite[Theorem 4.7]{AIN}. 
We also fix the operator $\partial \in \Diff_{2\V}^1(M)$ which is supported in $[0, \epsilon) \times Y$ and near $Y$ coincides
with $\partial_t$ for $t \in (-\epsilon, \epsilon)$. 

\begin{Lem}
For $j \in \Nn_0$ with $s > j - \frac{1}{2}$ we obtain
\[
\partial_{+}^j \colon H_{\V}^s(X) \to H_{\V}^{s-j}(X)
\]

is continuous. 
\label{Lem:partial}
\end{Lem}

\begin{proof}
The extension by zero operator $\chi^{0} \colon H_{\V}^s(X) \to H_{\V}^s(M)$ is a priori continuous for $\frac{1}{2} > s > -\frac{1}{2}$. 
Since any differential operator has the transmission property we obtain that $\partial_{+}$ is continuous for $s > j - \frac{1}{2}$ by \cite[Theorem 1.14]{HS}. 
\end{proof}

\begin{Def}
Set for a given $f \in C^{\infty}(X)$ with support in $Y_{(\epsilon)}$ 
\[
(\gamma_j f)(x') := \lim_{t \to 0} (\partial_{x_n}^j f)(x', t), \ (x', x_n) \in Y_{(\epsilon)}. 
\]

Given $u \in C_c^{\infty}(Y_0)$ denote by $u \otimes \delta_Y \in \D'(M_0)$ the distribution given by
\[
\scal{u \otimes \delta_Y}{v} = \scal{u}{v(-, 0)} = \int_{Y_0} u(x') v_{|Y_0}(x') \,d\nu(x'), \ v \in C_c^{\infty}(M_0). 
\]
\label{Def:gammaj}
\end{Def}

\begin{Prop}
For $s > j + \frac{1}{2}$ we have
\[
\gamma_j \colon H_{\V}^s(X) \to H_{\W}^{s - j - \frac{1}{2}}(Y) 
\]

is continuous. 
 
\label{Prop:gammaj}
\end{Prop}

\begin{proof}



The differential operator
\[
\partial^j \colon H_{\V}^s(M) \to H_{\V}^{s - j}(M)
\]
is continuous.
Denote by $e_s \colon H_{\V}^s(X) \to H_{2 \V}^s(M)$ a continuous extension operator depending on $s$.
By the previous remarks
\[
\gamma_Y \colon H_{2\V}^{s-j}(M) \to H_{2 \V}^{s -j -\frac{1}{2}}(M)
\] 

is continuous provided $s > j - \frac{1}{2}$. 
In the fixed tubular neighborhood we can write $\gamma_j = \gamma_Y \circ \partial^j \circ e_s$. 
This proves the claim by the continuity of $\gamma_Y, \ \partial^j$ and $e_s$.
\end{proof}

We are now in a position to define the operators $\B_{\V}^{m,d}(X, Y)$ of arbitrary order $m \in \Rr$ and type $d \in \Nn_0$. 

\begin{Def}
\emph{i)} A \emph{trace operator} of order $m \in \Rr$ and type $d \in \Nn_0$ is defined as $T \in \Trace_{\V}^{m,d}(X, Y)$ if 
there is a sequence $T_j \in \Trace_{\V}^{m-j,0}(X, Y)$ such that
\[
T = \sum_{j=0}^{d} T_j \partial_{+}^j. 
\]

A residual trace operator of type $d \in \Nn_0$ is defined in the same way for a sequence $T_j \in \Trace_{\V}^{-\infty, 0}(X, Y), \ 0 \leq j \leq d$.

\emph{ii)} A \emph{singular Green operator} of order $m \in \Rr$ and type $d \in \Nn_0$ is defined as $G \in \Green_{\V}^{m,d}(X, Y)$
if there is a sequence $G_j \in \Green_{\V}^{m-j,0}(X, Y)$ such that
\[
G = \sum_{j=0}^{d} G_j \partial_{+}^j. 
\] 

A residual singular Green operator of type $d \in \Nn_0$ is defined in the same way for a sequence $G_j \in \Green_{\V}^{-\infty, 0}(X, Y), \ 0 \leq j \leq d$.
\label{Def:Bmd}
\end{Def}

\begin{Rem}
In \cite{RS}, p. 149 we have the following standard representations for elements in Boutet de Monvel's calculus.
A trace operator $T \in \Trace_{\V}^{m,0}(X, Y)$ takes the form $T = \gamma_Y \circ Q_{+}$ for a $Q \in \Psi_{2\V, tr}^m(M)$
and a potential operator $K \in \K_{\V}^m(X, Y)$ takes the form $K = \tilde{Q}_{+}(- \otimes \delta_Y)$ for a $\tilde{Q} \in \Psi_{2\V, tr}^m(M)$. 
We should remark that in our framework we can prove these properties in the same way. 
By an elementary but tedious computation we can show directly that the formal adjoint of a trace operator is a potential operator also in the
order $m > 0$ cases. 
\label{Rem:stdrepr}
\end{Rem}

\begin{Thm}
\emph{i)} A trace operator $T \in \Trace_{\V}^{m,d}(X, Y)$ extends to a continuous linear operator
\[
T \colon H_{\V}^s(X) \to H_{\W}^{s - m -\frac{1}{2}}(Y) 
\]

for $s > d - \frac{1}{2}$. 

\emph{ii)} A potential operator $K \in \Pot_{\V}^m(X, Y)$ extends to a continuous linear operator 
\[
K \colon H_{\W}^{s - \frac{1}{2}}(Y) \to H_{\V}^{s-m}(X) 
\]

for $s > -\frac{1}{2}$. 

\emph{iii)} A singular Green operator $G \in \Green_{\V}^{m,d}(X, Y)$ extends to a continuous linear operator 
\[
G \colon H_{\V}^s(X) \to H_{\V}^{s -m}(X)
\]

for $s > d - \frac{1}{2}$. 

\label{Thm:cont2}
\end{Thm}

\begin{proof}
\emph{i)} Let $T \in \Trace_{\V}^{m,d}(X, Y)$ and decompose $T$ as follows
\[
T = T_0 + T_d
\]

where $T \in \Trace_{\V}^{m,0}(X, Y)$ and $T_d$ can be written 
\[
T_d = \sum_{j=0}^{d-1} S_j \gamma_j, \ S_j \in \Psi_{\W}^{m-j}(Y)
\]

(see e.g. \cite{SS} for the proof). 

Hence the continuity follows from Proposition \ref{Prop:gammaj}. 

\emph{ii)} This follows because $K$ is the formal adjoint of a trace operator $T$. 

\emph{iii)} By the same proof as in \cite{SS}, Lemma 2.2.14 we can write
\[
G = \sum_{j=0}^{d-1} K_j \gamma_j + G_0, \ K_j \in \Pot_{\V}^{m-j, 0}(X, Y). 
\]

The first term on the right hand side is continuous by Proposition \ref{Prop:gammaj} and \emph{ii)}. 
The second term has the following series representation by \cite{RS}, Lemma 7, p.149
\[
G_0 = \sum_{j=1}^{\infty} \lambda_j \tilde{K}_j \tilde{T}_j 
\]

for a sequence of potential operators $\tilde{K}_j$ and trace operators $\tilde{T}_j$. 
Hence the continuity follows by \emph{i)} and \emph{ii)}. 
\end{proof}

In the sequel we will study the corresponding Guillemin completion for Boutet de Monvel's calculus. 
The completion has favorable algebraic properties. Most importantly, inverses (if they exist) are contained in the calculus. Define the completion $\overline{\B}_{\V}^{-\infty,0}(X, Y)$ of the residual Boutet de Monvel operators with regard to the family of norms
of operators $\L\left(\begin{matrix} H_{\V}^t(X) \\ \oplus \\ H_{\W}^t(Y) \end{matrix}, \ \begin{matrix} H_{\V}^r(X) \\ \oplus \\ H_{\W}^r(Y) \end{matrix}\right)$ on Sobolev spaces.

\begin{Def}
Let 
\[
\Bc_{\V}^{m,d}(X, Y) := \B_{\V}^{m,d}(X, Y) + \Bc^{-\infty, d}(X, Y) 
\]

be the \emph{completed} calculus of Boutet de Monvel operators.
\label{Def:completed}
\end{Def}

\begin{Prop}
Let $R_{+}^m := \chi^{+} R_{-}^m \chi^{0}$ with $R_{-}^m \in \overline{\Psi}_{tr, \V}^m(M)$ and $\tilde{R}^m \in \overline{\Psi}_{\W}^m(Y)$ order
reductions in the pseudodifferential calculi. 
Define $\Lambda^m := \diag(R_{+}^m, \tilde{R}^m)$, then $\Lambda^m \in \Bc_{\V}^{m,0}(X, Y)$. 
\label{Prop:red}
\end{Prop}

\begin{Prop}
Let $m \leq 0$, then $A \in \B_{\V}^{m,0}(X, Y)$ implies that $A^{\ast} \in \B_{\V}^{m,0}(X, Y)$ where
\[
A^{\ast} = \begin{pmatrix} \chi^{+} \op(a)^{\ast} \chi^{0} + G^{\ast} & T^{\ast} \\
            K^{\ast} & S^{\ast}
           \end{pmatrix}.
\]
\label{Prop:adjoint}
\end{Prop}

\begin{proof}
We have $[\chi^{+} \op(a) \chi^{0}]^{\ast} = \chi^{+} \op(a)^{\ast} \chi^{0}$ immediately because $m \leq 0$ and $L^2$-continuity. 
The adjoint of a trace operator is a potential operator and vice versa and the adjoint of a singular Green operator is a again a singular
Green operator by definition. 
Finally, $S^{\ast} \in \Psi_{\W}^m(Y)$ since the Lie calculus is closed under adjoints, see \cite[Corollary 3.3]{ALN}.
\end{proof}

For the following Theorem we use a common trick based on order reductions, see for example \cite[Theorem 6.3.6]{UB}. 

\begin{Thm}
Let $A \in \B_{\V}^{m,d}(X, Y)$ with $d \leq m_{+}$. 

\emph{i)} With $m \leq 0$ and $d = 0$ it follows that 
\[
A \cdot A^{\ast} \in \B_{\V}^{2m,0}(X, Y). 
\]

\emph{ii)} For $m > 0$ it follows that
\[
(A \cdot \Lambda^{-m}) \cdot (A \cdot \Lambda^{-m})^{\ast} \in \B_{\V}^{0,0}(X, Y). 
\]

\emph{iii)} If $A \in \Bc_{\V}^{m,d}(X, Y)$ is invertible for some $s \geq m$ as an operator
\[
A \colon \begin{matrix} H_{\V}^s(X, E) \\ \oplus \\ H_{\W}^{s - \frac{1}{2}}(Y) \end{matrix} \to \begin{matrix} H_{\V}^{s - m}(X, E) \\ \oplus \\ H_{\W}^{s - m -\frac{1}{2}}(Y, F) \end{matrix}
\]

then the inverse $A^{-1} \in \Bc_{\V}^{-m, \max\{s -m, 0\}}(X, Y)$.  
\label{Thm:inverse}
\end{Thm}

\begin{proof}
The properties \emph{i)} and \emph{ii)} follow immediately from Proposition \ref{Prop:adjoint}. 
In order to prove \emph{iii)} we consider the pseudodifferential order reductions 
\[
R_{1}^{-s} \colon L_{\V}^2(X, E) \oplus H_{\W}^{-\frac{1}{2}}(Y, F) \iso H_{\V}^s(X, E) \oplus H_{\W}^{s- \frac{1}{2}}(Y, F) 
\]

and
\[
R_2^{s-m} \colon H_{\V}^{s-m}(X, E) \oplus H_{\W}^{s- m- \frac{1}{2}}(Y, F) \iso L_{\V}^2(X, E) \oplus H_{\W}^{-\frac{1}{2}}(Y, F).
\]

of order $-s$ and $s -m$ respectively. 

Then $B := R_{2}^{s-m} A R_1^{-s}$ is contained in $\Bc_{\V}^{0,0}(X, Y)$ and is invertible. 
Therefore the assertion reduces to the order $(0,0)$ case. We apply the $\Psi^{\ast}$-property of the $(0,0)$-algebra which is
proved analogously to \cite[Theorems 6.1-6.2]{ALNV} and \cite[Lemma 4.8]{N}. 
\end{proof}


\section{Parametrix}
\label{parametrices}


In this section we will introduce the principal and principal boundary symbol of an operator in our calculus. 
We will define the notion of ellipticity and show that a parametrix exists under the previously stated conditions on the calculus.
A major technical problem is that in the Lie calculus already inverses of invertible operators are not necessarily contained.
This makes a parametrix construction difficult and we state here a version of such a result. 
There are at least two approaches to overcome the problem of inverses: \emph{(1)} using a larger calculus of pseudodifferential operators
(with asymptotics) or \emph{(2)} completing the algebra of pseudodifferential operators (non-canonically) such that inverses are contained.
In the last section we outlined the second approach. We will henceforth consider the completed Boutet de Monvel calculus as defined in Section \ref{properties}.

Fix the smooth, hermitian vector bundles $E_1, \ E_2 \to X, \ J_{\pm} \to Y$ and recall the notation for the boundary algebroid and its co-bundle
$\pi_{\partial} \colon \A_{\partial} \to Y, \ \overline{\pi}_{\partial} \colon \A_{\partial}^{\ast} \to Y$. 

We define the principal symbol and principal boundary symbol on $\B^{m,0}(\G^{+}, \X; \tilde{E}_1, \tilde{E}_2, \tilde{J}_{\pm})$
for $m \leq 0$.

\begin{itemize}
\item Set $T^q \X := \ker dq$ for the vertical tangent bundle over $\X$. 

\item Let $A = (A_x)_{x \in X} \in \B^{m,0}(\G^{+}, \X)$ be a $C^{\infty}$-family of Boutet de Monvel operators.

\item Then for each $x \in X$ we have a Boutet de Monvel operator $A_x \in \B_{prop}^{m,0}(\G_x^{+}, \X_x)$. 

\item Hence the principal symbol $\sigma(A_x)$ and the principal boundary symbol $\sigma_{\partial}(A_x)$ are defined
invariantly on $T^{\ast} \G_x^{+}$ and $T^{\ast} \X_x$ respectively.

\item By right-invariance of the family $A$ these symbols descend to a principal symbol $\sigma(A)$ and principal boundary symbol $\sigma_{\partial}(A)$
defined invariantly on $\A_{+}^{\ast}$ and $(T^q \X)^{\ast}$ respectively. 

\item Introduce the \emph{indicial symbol} $\R_F$ for a face $F \in \F_1(X)$ as the restriction $\R_F((A_x)_{x \in X}) = (A_y)_{y \in F}$. 
\end{itemize}

By noting that $T_{|Y}^q \X = \A_{\partial}$ we make the following definition for principal and principal boundary symbol
on the represented algebra $\B_{\V}^{m,0}(X, Y)$ for $m \leq 0$. 

\begin{Def}
To an element $A \in \B_{\V}^{m,0}(X, Y; E_1, E_2, J_{\pm})$ we associate the two principal symbols.

\emph{i)} The \emph{principal boundary symbol} $\sigma_{\partial}^{\V}(A)$ is defined as 
\[
\sigma_{\partial}^{\V}(A) := \sigma_{\partial|Y}(A). 
\]

This yields a section of the infinite dimensional bundle
\[
C^{\infty}(\A_{\partial}^{\ast}, \Hom(\overline{\pi}_{\partial}^{\ast} E_{1|Y} \otimes \SV, \overline{\pi}_{\partial}^{\ast} E_{2|Y} \otimes \SV)).
\]

Here $\SV \to \A_{\partial}^{\ast}$ is a bundle with fiber $S(\overline{\Rr}_{+})$ on the inward pointing normal direction.

In particular the restriction of $\sigma_{\partial}^{\V}$ to the interior $(X_0, Y_0)$ agrees with the principal boundary symbol on the interior.

\emph{ii)} The \emph{principal symbol} $\sigma^{\V}(A)$ which is the principal symbol of the pseudodifferential operator
in the upper left corner of the matrix $A$.
This yields a section in 
\[
C^{\infty}(\A_{+}^{\ast}, \Hom(E_1, E_2)). 
\] 

\emph{iii)} Define by $\Sigma_{\V}^{m,0}(X, Y; E_1, E_2, J_{\pm})$ the space consisting of pairs of principal symbols $(a, a_{\partial})$.
These are homogenous or $\kappa$-homogenous sections of the bundles $\A_{+}^{\ast}, \A_{\partial}^{\ast}$, respectively, with canonical compatibility condition.

\label{Def:bdysmb}
\end{Def}



%

\begin{Lem}
Given $A, B \in \B_{\V}^{0,0}(X, Y; E_1, E_2, J_{\pm})$ we have 
\[
\sigma^{\V}(A \cdot B) = \sigma^{\V}(A) \cdot \sigma^{\V}(B), \ \sigma_{\partial}^{\V}(A \cdot B) = \sigma_{\partial}^{\V}(A) \cdot \sigma_{\partial}^{\V}(B).
\]
\label{Lem:mult}
\end{Lem}

\begin{proof}
From the assumption that the groupoid $\G$ is Hausdorff we obtain by the same argument as in \cite[p.11]{N} that the vector representation furnishes an isomorphism 

\[
\B_{\V}^{0,0}(X, Y; E_1, E_2, J_{\pm}) \cong \B^{0,0}(\G^{+}, \X; \tilde{E}_1, \tilde{E}_2, \tilde{J}_{\pm}). 
\]

Since the principal and principal boundary symbol are each defined invariantly on the bundles $\A_{+}^{\ast}, \ \A_{\partial}^{\ast}$
the computation reduces to the equivariant families of Boutet de Monvel operators.
Hence given $A = (A_x)_{x \in X}, \ B = (B_x)_{x \in X}$ in $\B^{0,0}(\G^{+}, \X; \tilde{E}_1, \tilde{E}_2, \tilde{J}_{\pm})$ we have
\begin{align*}
\sigma_{\partial}(A \cdot B) &= \sigma_{\partial}((A_x \cdot B_x)_{x \in X}) = (\sigma_{\partial}(A_x \cdot B_x))_{x \in X} \\
&= (\sigma_{\partial}(A_x) \cdot \sigma_{\partial}(B_x))_{x \in X} = \sigma_{\partial}(A) \cdot \sigma_{\partial}(B).
\end{align*}

In the same way we obtain multiplicativity of the principal symbol. 
\end{proof}

Since we will in the following only be concerned with represented operators we will simply write $\sigma$ and $\sigma_{\partial}$ for $\sigma^{\V}$ and $\sigma_{\partial}^{\V}$.

For the following result and proof in the standard case see e. g. \cite{RS}.

\begin{Thm}
The following sequence is exact
\[
\xymatrix{
\B_{\V}^{-1,0}(X, Y; E_1, E_2, J_{\pm}) \ar@{>->}[r] & \B_{\V}^{0,0}(X, Y; E_1, E_2, J_{\pm}) \ar@{->>}[r]^{\sigma \oplus \sigma_{\partial}} & \Sigma_{\V}^{0,0}(X, Y; E_1, E_2, J_{\pm}).
}
\]

\label{Thm:exact}
\end{Thm}

\begin{proof}
\emph{i)} The same exact sequence holds for the interior calculus. Since the principal symbols are extensions of the interior
we immediately obtain that $\ker \sigma \oplus \sigma_{\partial} = \ker \sigma \cap \ker \sigma_{\partial} = \B_{\V}^{-1,0}(X, Y)$. 

\emph{ii)} To prove surjectivity let $(a, a_{\partial}) \in \Sigma_{\V}^{0,0}$.

Since we also have an analogous exact sequence for the class of pseudodifferential operators $\Psi_{tr,2\V}^m$ it suffices to find singular Green, trace and potential operators in the preimage. 

In a fixed small tubular neighborhood $\U \cong Y_{(\epsilon)}$ trivialize the singular normal bundle $\N$. 
Let $\{U_i\}$ be a normal covering of $Y$ (assumed finite by compactness of $Y$) and let $\{\varphi_i\}$ be a subordinate partition of unity.
We also fix a boundary defining function $\rho_Y \colon M \to \Rr$ such that $\{\rho_Y = 0\} = Y$. This is defined by the tubular neighborhood
theorem for Lie manifolds (cf. \cite{AIN}). 
Hence for the diffeomorphism of tubular neighborhoods $\nu \colon Y \times (-\epsilon, \epsilon) \to \U \subset M$ we have 
\[
(\rho_Y \circ \nu)(x', x_n) = x_n, \ x' \in Y, \ x_n \in (-\epsilon, \epsilon).
\]

Let $\varphi \in C^{\infty}(\overline{\Rr}_{+})$ be a cutoff function such that $\varphi(x_n) = 0$ close to $0$. 
Hence we can locally on each trivialization $\N_{|U_i}^{+} \cong U_i \times \overline{\Rr}_{+}$ construct singular Green, potential and trace operators.
Via our partition of unity we obtain corresponding global symbols.


This furnishes a global right inverse morphism $\OpV \colon \Sigma_{\V}^{0,0} \to \B_{\V}^{0,0}$. 
\end{proof}

The notion of ellipticity we introduce here is the usual condition of Shapiro-Lopatinski type. It is sufficient to obtain a parametrix. 
We obtain Fredholm operators if we additionally request the invertibility of the indicial symbols.

\begin{Def}
An operator $A \in \Bc_{\V}^{m,0}(X, Y)$ is called

\emph{i)} \emph{$\V$-elliptic} iff $(\sigma^{m} \oplus \sigma_{\partial})(A)$ is pointwise bijective,

\emph{ii)} \emph{elliptic} if $A$ is $\V$-elliptic and for each $F \in \F_1(X)$ the indicial symbol $\R_F(A)$ is invertible.
\label{Def:SL}
\end{Def}



\begin{Rem}
The elements of the symbol algebra $\Sigma_{\V}$ can be written more simply in terms of the \emph{action in the normal direction}.
Let us define this action for a given symbol $p \in S_{tr}^{m}(\A^{\ast})$.
We set
\[
(\op_n^{+} p)(x', \xi') = p(x', 0, \xi', D_n)_{+}
\]

then in local trivializations of the bundles $E_1, \ E_2$ and $\SV$ we have
\[
(\op_n^{+} p)(x', \xi') \colon S(\overline{\Rr}_{+}) \otimes \Cc^k \to S(\overline{\Rr}_{+}) \otimes \Cc^k
\]

where $k$ is the fibre dimension of $E_1, \ E_2$.

We will alternatively consider these maps as operators acting on $L^2(\overline{\Rr}_{+})$, the closure of $S(\overline{\Rr}_{+})$.
In \cite{RS} the action in the normal direction is defined for the boundary symbols which act more generally as \emph{Wiener-Hopf operators} on the spaces $H^{+}$
and their respective $L^2$-closures. We will not need this formulation in the present work.

Denote by $p \colon S^{\ast} \A_{\partial} \to Y$ the canonical projection of the sphere bundle. 

Let $A$ be $\V$-elliptic and consider the principal symbol $\sigma_P$ of the pseudodifferential operator $P$ in the upper left corner.
Since the infinite dimensional bundles are trivial the action in the normal direction
\[
\op_n^{+} \sigma_P \colon p^{\ast} E_1 \otimes \SV \to p^{\ast} E_2 \otimes \SV
\]

yields a Fredholm family, parametrized over $S^{\ast} \A_{\partial}$ which preserves the index, cf. \cite{RS}, Prop. 5, p.95.

We can express the principal boundary symbol (with analogously defined actions in the normal direction for Green, potential and trace operators)
\[
\sigma_{\partial}(A) = \begin{pmatrix} \op_n^{+} p + \op_n g & \op_n k \\
\op_n t & s \end{pmatrix}
\]

as an element of $C^{\infty}(S^{\ast} \A_{\partial}, \Hom(p^{\ast} E_{1|Y} \otimes \SV, p^{\ast} E_{2|Y} \otimes \SV)$. 


\label{Rem:SL}
\end{Rem}

\begin{Proof}[of Thm. \ref{Thm:parametrix}]
\emph{a)} First we have to show that the symbol algebra (for the completed calculus) $\overline{\Sigma}_{\V}$ is closed
under inverse.
Specifically, let $A \in \Bc_{V}^{0,0}(X, Y)$ be $\V$-elliptic and invertible. 
Denote by $P$ the pseudodifferential operator in the upper left corner of $A$. By $\V$-ellipticity of $A$ we fix
a $Q \in \overline{\Psi}_{\V}^{0}(M)$ such that $\sigma(P^{-1}) = \sigma(Q)$. 
Set $Q = \op(q)$ for a symbol $q \in S^{0}(\A^{\ast})$. 
We want to show that there is a an elliptic boundary symbol of the form
\[
\sigma_{\partial}(B) = \begin{pmatrix} \op_n^{+} q + \op_n g & \op_n k \\
\op_n t & s \end{pmatrix} \colon \begin{matrix} p^{\ast} E_1 \otimes \SV \\ \oplus \\ p^{\ast} J_{+} \end{matrix} \to \begin{matrix} p^{\ast} E_2 \otimes \SV \\ \oplus \\ p^{\ast} J_{-} \end{matrix}
\] 

such that $\sigma_{\partial}(B) = \sigma_{\partial}(A)^{-1}$. 

For this we can adapt an approach to the problem the idea of which goes back to Boutet de Monvel, \cite{BM} (the index bundle, p.35) and which can be found in detail in the reference \cite{RS}. 
We divide the argument into three parts.

\emph{I)} We have
\[
\dim \ker \op_n^{+} \sigma_P \leq \text{const}, \ \dim \coker \op_n^{+} \sigma_P \leq \text{const}
\]

with constants independent of the parameter $\varrho \in S^{\ast} \A_{\partial}$. 
Then there is a finite dimensional trivial subbundle $W$ in $p^{\ast} E \otimes \SV$ such that
\[
(\im \sigma_P)_{(x', \xi')} + W_{(x', \xi')} = (p^{\ast} E_2 \otimes \SV)_{(x', \xi')}
\]

for each $(x', \xi') \in S^{\ast} \A_{\partial}$. 
We obtain the index element
\[
\ind_{S^{\ast} \A_{\partial}} \op_n^{+} \sigma_P \in K(S^{\ast} \A_{\partial})
\]

which depends on the homotopy class of the Fredholm family.
Let $(\sigma^{(t)})_{t \in [0,1]}$ be a homotopy of elliptic symbols in $S^{0}(\A^{\ast})$, then
\[
\ind_{S^{\ast} \A_{\partial}} \op_n^{+} \sigma^{(0)} = \ind_{S^{\ast} \A_{\partial}} \op_n^{+} \sigma^{(1)}. 
\]

Hence for $\sigma_P$ we obtain
\[
\ind_{S^{\ast} \A_{\partial}} \op_n^{+} \sigma_P = [p^{\ast} J_{+}] - [p^{\ast} J_{-}].
\]

By the same argument as in \cite{RS}, Prop. 11 on page 199 it follows that
\begin{align}
& \ind_{S^{\ast} \A_{\partial}^{\ast}} \op_n^{+} p \in p^{\ast} K(Y). \label{Fhpar}
\end{align}

\emph{Claim:} By \eqref{Fhpar} there is a Green symbol $g_0$ and bundles $\tilde{J}_{+}, \ \tilde{J}_{-} \to Y$ such that
\begin{align*}
&\ker_{S^{\ast} \A_{\partial}}(\op_n^{+} \sigma_P + \op_n g_0) \cong p^{\ast} \tilde{J}_{+}, \\
& \coker_{S^{\ast} \A_{\partial}}(\op_n^{+} \sigma_P + \op_n g_0) \cong p^{\ast} \tilde{J}_{-}. 
\end{align*}

For the proof see \cite{RS}, p. 201. 

\emph{II)} The construction of the boundary symbol is a parameter dependent construction of $C^{\infty}$-boundary symbol (an isomorphism)
\[
b_{\partial} = \begin{pmatrix} \op_n^{+} q + \op_n g_1 & \op_n k_1 \\
\op_n t_1 & s_1 \end{pmatrix} \colon \begin{matrix} p^{\ast} E_1 \otimes \SV \\ \oplus \\ p^{\ast} J_{+} \end{matrix} \to \begin{matrix} p^{\ast} E_2 \otimes \SV \\ \oplus \\ p^{\ast} J_{-} \end{matrix}.
\]

Since $\op_n^{+} q$ is a parametrix of $\op_n^{+} p$ in the sense of Fredholm families we have by part \emph{I)}
\[
\ind_{S^{\ast} \A_{\partial}} \op_n^{+} q = [p^{\ast} J_{+}] - [p^{\ast} J_{-}]. 
\]

In order to find an isomorphism of the form $b_{\partial}$ it suffices to find a Green symbol $g_1$ with
\begin{align*}
& \ker_{S^{\ast} \A_{\partial}} (\op_n^{+} q + \op_n g_1) \cong p^{\ast} J_{+}, \\
& \coker_{S^{\ast} \A_{\partial}} (\op_n^{+} q + \op_n g_1) \cong p^{\ast} J_{-}. 
\end{align*}

We know by the claim in part \emph{I)} that there is a Green symbol $g_0$ with 
\begin{align*}
& \tilde{J}_0^{+} = \ker_{S^{\ast} \A_{\partial}} (\op_n^{+} q + \op_n g_0) \cong p^{\ast} J_0^{+}, \\
& \tilde{J}_0^{-} = \coker_{S^{\ast} \A_{\partial}} (\op_n^{+} q + \op_n g_0) \cong p^{\ast} \Cc^N
\end{align*}

for suitable vector bundles $J_{0}^{\pm} \to Y$ and $N$ such that 
\[
[p^{\ast} J_{0}^{+}] - [\Cc^N] = [p^{\ast} J_{+}] - [p^{\ast} J_{-}]
\]

by Lemma 5, p. 203 \cite{RS}.

For $N$ sufficiently large there are decompositions 
\begin{align*}
& \tilde{J}_{0}^{+} = \tilde{V} \oplus \tilde{J}_{+}, \ \tilde{J}_0 = \tilde{W} \oplus \tilde{J}_{-}, \\
& J_{0}^{+} \cong V \oplus J_{+}, \ \Cc^N \cong W \oplus J_{-}
\end{align*}

for vector bundles $V, W \to Y$ with isomorphisms
\[
\tilde{V} \cong p^{\ast} V, \ \tilde{J}_{+} \cong p^{\ast} J_{+}, \ \tilde{W} \cong p^{\ast} W, \ \tilde{J}_{-} \cong p^{\ast} J_{-}
\]

where $\op_n^{+} q + \op_n g_0$ induces an isomorphism $\beta \colon \tilde{W} \to \tilde{V}$. 
There is a Green symbol $g_2$ so that the isomorphism $\beta$ is induced by $-\op_n g_2$ and that $\op_n g_2$ vanishes 
on the complement of $\tilde{W}$ in $p^{\ast} E \otimes \SV$. 
Then $g_1 = g_0 + g_2$ is a Green symbol with the desired property. 

\emph{III)} The inverse of $a_{\partial} = \sigma_{\partial}(A)$ can be calculated separately for each $\varrho \in S^{\ast} \A_{\partial}$ (cf.
2.1.2.4, Prop. 6, p. 110).
What is left to show is the smoothness of the inverse symbol depending on $\varrho \in S^{\ast} \A_{\partial}$. 
For $b_{\partial}$ the smoothness is clear. Additionally, the composition of smooth symbols yields smooth symbols.
Thus $c_{\partial} = a_{\partial} b_{\partial}$ is in $C^{\infty}(S^{\ast} \A_{\partial})$. 
Hence if we can show that $c_{\partial}^{-1}$ is $C^{\infty}$ we obtain that $a_{\partial}^{-1} = b_{\partial} c_{\partial}^{-1}$ is $C^{\infty}$. 
This is a lenghy but elementary calculation for which we refer to Prop. 6, p. 204-206 in \cite{RS}.

The parametrix is obtained as follows.

\emph{i)} From the exact sequence given in Theorem \ref{Thm:exact} there is a $B \in \Bc_{\V}^{0,0}(X, Y)$ such that
\[
(\sigma \oplus \sigma_{\partial})(B) = (\sigma \oplus \sigma_{\partial})(A)^{-1}. 
\]

Then by the multiplicativity of the principal symbol and the principal boundary symbol it follows that
\[
(\sigma \oplus \sigma_{\partial})(I - AB) = 0.
\]

Applying the exact sequence once more we have that $R := I - AB \in \Bc_{\V}^{-1,0}(X, Y)$. 
Hence $B$ is a right parametrix of $A$ of order $1$. 
Setting $B_k = B(I + R + \cdots + R^{k-1})$ we obtain 
\[
A B_k = (I - R) (I + R + \cdots + R^{k-1}) = I - R^k
\]

with $R^k \in \Bc_{\V}^{-k,0}(X, Y)$. 
Thus $B_k$ is a right parametrix of $A$ of order $k$. 

\emph{ii)} Let $(B_k)_{k \in \Nn_0}$ be a sequence of right parametrices of $A$ of orders $k \in \Nn_0$. 
By asymptotic completeness we can find $B \sim \sum_{i} B_i$ such that $AB - I \in \B_{\V}^{-\infty, 0}(X, Y)$. 
Hence we have found a right parametrix up to residual terms.

\emph{iii)} Fix a right parametrix up to residual terms $B_1$ of $A$. Analogously to \emph{i)} and \emph{ii)} we can find a left parametrix $B_2$ of $A$ such that
\[
I - A B_1 = R_1 \in \Bc_{\V}^{-\infty, 0}(X, Y) \ \text{and} \ I - B_2 A = R_2 \in \Bc_{\V}^{-\infty, 0}(X, Y). 
\]

Rewrite the operator $A B_2 A B_1$ as follows
\[
A B_2 A B_1 = A B_2 (I - R_1) = A (I - R_2) B_1 = A B_1 - A R_2 B_1 = (I - R_1) - A R_2 B_1
\]

hence $A B_2 A B_1 + A R_2 B_1 = I - R_1$ so that 
\[
(A B_2 A + A R_2) B_1 = I - R_1.
\]

The left hand side of the previous equation equals 
\[
A B_2 (I - R_1) + A R_2 B_1
\]

and note that this is $\equiv \ A B_2  \mod  \Bc_{\V}^{-\infty, 0}(X, Y)$.

Hence $B_2$ is also a right parametrix equal up to residual terms to $I - R_1$. 
Analogously one proves that $B_1$ is also a left parametrix equal up to residual terms to $I - R_2$. 
It follows that $B_1 \equiv B_2  \mod  \Bc_{\V}^{-\infty, 0}(X, Y)$ and therefore there is an operator $B \in \Bc_{\V}^{0,0}(X, Y)$ equal to $B_1 \mod \Bc_{\V}^{-\infty,0}(X, Y)$ and $B_2 \mod \Bc_{\V}^{-\infty}(X, Y)$ such that
\[
I - AB \in \Bc_{\V}^{-\infty, 0}(X, Y), \ \text{and} \ I - BA \in \Bc_{\V}^{-\infty, 0}(X, Y). 
\]

\emph{b)} Consider the pullback algebra $\Sigma_{BM}$ which is defined as the restricted direct sum of $\Sigma_{\V}$ and the indicial algebras $\oplus_{F \in \F_1(X)} \B_{\V(F)}^{0,0}(F, F_{reg})$.
We can summarize the situation in the following pullback diagram.

\[
\xymatrix{
\Sigma_{BM} \ar@{->>}[d(1.8)]_{\pi_2} \ar@{->>}[r(1.7)]^-{\pi_1} & &  \overline{\Sigma}_{\V}^{0,0} \ar@{->>} [d(1.8)]^{(r_F)_{F \in \F_1(M)}} & & \\
& \ar@{->>}[ul]_{\sigma \oplus \sigma_{\partial} (\oplus_{F} \R_F)} \ar@{->>}[dl]_{\oplus_{F} \R_F} \overline{\B}_{\V}^{0,0} \ar@{->>}[ur]^{\sigma \oplus \sigma_{\partial}} \ar@{->>}[dr]_{} & \\
\oplus_{F \in \F_1(X)} \overline{\B}_{\V(F)}^{0,0}(F, F_{reg}) \ar@{->>}[r(1.4)]^-{\oplus_{F \in \F_1(X)} \sigma_F \oplus \sigma_{\partial|F}} & &  \oplus_{F \in \F_1(X)} \overline{\Sigma}_{\V(F)}^{0,0}(F, F_{reg}) & &
}
\]

We will generalize the Fredholm conditions given in the special case of the pseudodifferential operators on Lie manifolds, see e.g. \cite{N}.

\emph{Claim:} The following sequence is exact
\[
\xymatrix{
\K(L_{\V}^2(X)) \oplus \K(L_{\W}^2(Y)) \ar@{>->}[r] & \B_{\V}^{0,0}(X, Y) \ar@{->>}[r(2.0)]^-{\sigma \oplus \sigma_{\partial} (\oplus_F \R_F)} & & & \Sigma_{BM}.
}
\]

The surjectivity is immediate. Let $A \in \B_{\V}^{0,0}(X, Y)$, then we prove that
\[
A \colon \begin{matrix} L_{\V}^2(X) \\ \oplus \\ L_{\W}^2(Y) \end{matrix} \to \begin{matrix} L_{\V}^2(X) \\ \oplus \\ L_{\W}^2(Y) \end{matrix}
\]

is compact if and only if $\sigma(A) = 0, \ \sigma_{\partial}(A) = 0$ and $\R_F(A) = 0$ for each $F \in \F_1(X)$. 

Assume that $\sigma \oplus \sigma_{\partial} \oplus (\oplus_F \R_F)(A) = 0$. 
From the short exact sequence in Theorem \ref{Thm:exact} it follows that $A \in \B_{\V}^{-1,0}(X, Y)$.
Let $(\rho_F)_{F \in \F_1(X)}$ be the collection of boundary defining functions of the faces at infinity of $X$.
Setting
\[
\rho := \prod_F \rho_F
\]

we have that $A = \rho B$ for some $B \in B_{\V}^{-1,0}(X, Y)$.
Hence 
\[
A \colon \begin{matrix} L_{\V}^2(X) \\ \oplus \\ L_{\W}^2(Y) \end{matrix} \to \begin{matrix} \rho H_{\V}^1(X) \\ \oplus \\ H_{\W}^1(Y) \end{matrix}
\]

is continuous.

By the generalized Kondrachov's theorem, see \cite[Theorem 3.6]{AIN}, we obtain that $\rho H_{\V}^1(X) \hookrightarrow L_{\V}^2(X)$
is compact. Also $H_{\W}^1(Y) \hookrightarrow L_{\W}^2(Y) = H_{\W}^{0}(Y)$ is compact. 
We obtain that
\[
A \colon \begin{matrix} L_{\V}^2(X) \\ \oplus \\ L_{\W}^2(Y) \end{matrix} \to \begin{matrix} L_{\V}^2(X) \\ \oplus \\ L_{\W}^2(Y) \end{matrix}
\]

is compact. 

For the other direction assume that $A$ is compact. Then $\sigma(A) = 0$ and $\sigma_{\partial}(A) = 0$.
Assume for contradiction that there is a face $F \in \F_1(X)$ such that $\R_F(A) \not= 0$ and fix such an $F$.
By the Hausdorff and amenability property of $\G$ we can identify bijectively $\B_{\V}^{0,0}(X, Y) \cong \B^{0,0}(\G^{+}, \X)$ via the vector
representation, see \cite[p.11]{N}. 
The restriction $\R_F$ is defined on the level of families by $(A_x)_{x \in X} \mapsto (A_y)_{y \in F}$ and by
the Hausdorff property of the groupoid $\G$ we have that $X \ni x \mapsto \left\|A_x \begin{pmatrix} \varphi_x \\ \psi_x \end{pmatrix}\right\|$
is a continuous function for each $\varphi \in C_c^{\infty}(\G^{+}), \ \psi \in C_c^{\infty}(\X)$. 
Let $x \in F$ and $\varphi \in C_c^{\infty}(\G_x^{+}), \ \psi \in C_c^{\infty}(\X_x)$ be given such that $A_x \begin{pmatrix} \varphi \\ \psi \end{pmatrix} \not= 0$
and let $\tilde{\varphi}, \ \tilde{\psi}$ be extensions to functions in $C_c^{\infty}(\G^{+})$ and $C_c^{\infty}(\X)$ respectively.
Then we have $\tilde{\varphi}_y = \varphi_{|\G_y}, \ \tilde{\psi}_y = \psi_{|\X_y}$ for $y \in X_0$.
These restrictions identify as functions $\tilde{\varphi}_y \in C_c^{\infty}(X_0)$ and $\tilde{\psi}_y \in C_c^{\infty}(Y_0)$ 
via the canonical isometries $\G_y^{+} \cong X_0, \ \X_y \cong Y_0$. 
The supports of $\tilde{\varphi}_y, \ \tilde{\psi}_y$ increase via $y \to x$ for $y \in X_0$. 
This implies weak convergence in Sobolev spaces $\tilde{\varphi}_y \to 0, \ \tilde{\psi}_y \to 0$. 
By continuity it follows
\[
A_y \begin{pmatrix} \varphi_y \\ \psi_y \end{pmatrix} \to A_x \begin{pmatrix} \varphi_x \\ \psi_x \end{pmatrix} \not= 0. 
\]

This yields a contradiction to the compactness of $A$. 
Hence the proof of the claim is finished.

Set $\H := L_{\V}^2(X) \oplus L_{\W}^2(Y)$ and denote by $\C(\H) := \L(\H) / \K(\H)$ the Calkin algebra with quotient
map $q_{\C} \colon \L(\H) \to \C(\H)$. 
We want to show that $A$ is elliptic if and only if $A$ is Fredholm on the Hilbert space $\H$.
By amenability property of the groupoid we know that the vector representation furnishes an injective $\ast$-representation (cf. \cite[Theorem 4.4]{LN})
\[
\pi \colon \B_{\V}^{0,0}(X, Y) \to \L(\H). 
\]

An operator $A \in \L(\H)$ is Fredholm if and only if $q_{\C}(A) \in \C(\H)$ is invertible.
Apply the standard exact sequence for the indicial symbol for order $-1$ (note that by invariance of each singular hyperface $F$ the indicial symbol is surjective)
\[
\xymatrix{
\ker \R_F^{-1} \ar@{>->}[r] & \B_{\V}^{-1,0}(X, Y) \ar@{->>}[r]^-{\R_F^{-1}} & \B_{\V}^{-1,0}(F, F_{reg}) =: \B_{F}^{-1,0}.
}
\]

Consider the diagram
\[
\xymatrix{
\oplus_{F} \ker \R_F^{-1} \ar@{>->}[dr] \ar@{>->}[d] \ar@{==}[r] & \K(\H) \ar@{>->}[d] \ar@{>->}[r] & \L(\H) \ar@{->>}[r]^{q_{\C}} & \C(\H) \\
\B_{\V}^{-1,0} \ar@{->>}[d] \ar@{>->}[r] & \B_{\V}^{0,0} \ar@{->>}[d]^-{\sigma \oplus \sigma_{\partial} \oplus (\oplus_F \R_F)} \ar@{>->}[ur]^{\pi} \ar@{->>}[r]^{\sigma \oplus \sigma_{\partial}} & \Sigma_{\V} & \\
\oplus_F \B_{F}^{-1,0} & \Sigma_{BM} & &
}
\]
One can check the equality ($A$ being SL-elliptic, see also \cite[Theorem 9]{LN})
\[
\sigma_{ess}(\pi(A)) = \bigcup_{F \in \F_1(X)} \sigma(\R_F(A)) \cup \bigcup_{\xi \in S^{\ast} \A} \mathrm{spec}(\sigma(A)(\xi)) \cup \bigcup_{\xi' \in S^{\ast} \A_{\partial}} \mathrm{spec}(\sigma_{\partial}(A)(\xi'))
\]

where $\sigma_{ess}$ denotes the essential spectrum and $\mathrm{spec}(\sigma_{\partial}(A)(\xi'))$ denotes the spectrum of the matrix defined by the operator
valued symbol $\sigma_{\partial}(A)$ at the point $\xi'$. 

Hence $\pi$ induces an injective $\ast$-homomorphism
\[
\B_{\V}^{0,0}(X, Y) / \K(\H) \to \C(\H).
\]

Thus $q_{C}(A)$ is invertible in $\C(\H)$ if and only if $(\R_F \oplus \sigma \oplus \sigma_{\partial})(A)$ is
pointwise invertible for each $F \in \F_1(X)$. \qed
\end{Proof}

\end{document}